\documentclass[11pt]{amsart}
\usepackage{amssymb,latexsym, amsmath, amsxtra}
\usepackage[dvips]{graphics}
\usepackage{epsfig}

\theoremstyle{plain}
\newtheorem{theorem}{Theorem}[section]
\newtheorem{corollary}[theorem]{Corollary}
\newtheorem{conjecture}[theorem]{Conjecture}

\newtheorem{lemma}[theorem]{Lemma}

\theoremstyle{definition}

\theoremstyle{remark}

\numberwithin{equation}{section}

\begin{document}

\title{On the orthogonal symmetry of $L$-functions of a family of Hecke Gr\"{o}ssencharacters}

\abstract 
The family of symmetric powers of an $L$-function associated with an
elliptic curve with complex multiplication has received much attention from algebraic, automorphic and p-adic points of view.
Here we examine one explicit such family from the perspectives of classical analytic number theory and random matrix theory, 
especially focusing on evidence for the  symmetry type of the family.
In particular,
we investigate the values at the central point and give
evidence that this family can be modeled by ensembles of
orthogonal matrices.  We prove an asymptotic formula with power savings for the average of these L-values, 
which reproduces, by a completely different method,  an asymptotic formula proven by Greenberg and Villegas--Zagier. We give an upper bound 
for the second moment which is conjecturally too large by just one logarithm.  We also give an explicit conjecture for the second
moment of this family, with power savings. Finally, we compute the one level density for this family with a test function whose Fourier transform 
has limited support. It is known by the work of Villegas -- Zagier  that the subset of these  L-functions from our family which have even functional
equations never vanish; we show to what extent this result is reflected by our analytic results. 
 
\endabstract

\author {J.B. Conrey}
\address{American Institute of Mathematics,
360 Portage Ave, Palo Alto, CA 94306 USA and School of
Mathematics, University of Bristol, Bristol, BS8 1TW, United
Kingdom} \email{conrey@aimath.org}

\author{N.C. Snaith}
\address{School of Mathematics,
University of Bristol, Bristol, BS8 1TW, United Kingdom}
\email{N.C.Snaith@bris.ac.uk}



\thanks{
Research of the first author supported by the American Institute
of Mathematics and by a grant from the National Science
Foundation. The second author was partially supported by an EPSRC
Advanced Research Fellowship and subsequently sponsored by the Air
Force Office of Scientific Research, Air
  Force Material Command, USAF, under grant number FA8655-10-1-3088.} \maketitle

\tableofcontents

\section{Introduction}
\label{sec:intro}

$L$-functions are fundamental objects in number theory  that carry
a lot of arithmetic information.  Probably the most famous example
is the Birch and Swinnerton-Dyer conjecture that equates the rank
of an elliptic curve with the order of vanishing of its
$L$-function at the central point.  It is generally believed that
the vanishing of an $L$-function at its central point indicates
some arithmetic-geometric structure.  There are many theorems
concerning the first-order vanishing of elliptic curve
$L$-functions  and random matrix theory has been used to model the
frequency
 of second-order vanishing \cite{kn:ckrs00}.  In addition,
 the Langlands philosophy predicts that
for any $L$-function arising from an automorphic representation
there is a new $L$-function associated with the $r$th symmetric
power representation. Combining these ideas, Barry Mazur asked the
following question: Given the $L$-function of an elliptic curve
$E/\mathbb Q$, is it true that the central value of the
$L$-function of its $n$th symmetric power vanishes, if ever, for
at most finitely many values of $n$? He admitted that it would
likely be too difficult to answer this question, but further asked
if random matrix theory could provide a model for this question.

We investigated this interesting question and quickly agreed that
it was much too difficult to answer. At this stage, we cannot even
decide whether, generically,  the collection of $\{L(\mbox{sym}^n
E,s)\}$ constitutes a family in the sense of
\cite{kn:katzsarnak99a} or \cite{kn:cfkrs}. But we did address his
question in the interesting special case that $E$ is an elliptic curve with
complex multiplication. In this situation, $L(\mbox{sym}^n E,s)$
is no longer a primitive $L$-function, i.e. it factors into other
$L$-functions. Basically $ L(\mbox{sym}^n E,s)$ is a product of
$L$-functions associated with cusp forms of increasing weight, see  the discussion at equation (\ref{eq:factoring}).  If
one of these $L$-functions happens to vanish at its central point,
then, since that $L$-function will show up in the factorizations
of infinitely many of the $L(\mbox{sym}^n E,s)$ we have a somewhat
trivial answer to Mazur's question.  A better question in this
case is whether infinitely often the new primitive part of each of
$L(\mbox{sym}^n E,s)$ can vanish at its central point. 
This is a question that has received attention from an automorphic and p-adic perspective, notably in works of 
Gross-Zagier \cite{kn:grozag80}, Villegas - Zagier \cite{kn:vilzag93} and  Greenberg \cite{kn:gre83}. 
In particular it is known in some instances that when the functional equation has even 
type there is never any vanishing. This is far better than what one could hope for by analytic methods,
where the best one could possibly  achieve would be non-vanishing in 100\% of the cases. 
Nevertheless, in view of the success by algebraic methods, it is interesting to compare what happens 
with a  classical approach.

With this classical method we have had some partial success. 
In particular, the primitive
parts alluded to above do seem to form an orthogonal family, and we can model this
family using random matrix theory. We can also take some theoretical steps  and in
particular can prove an asymptotic formula, with power savings,  for the first moment of
the $L$-functions in this family. This improves an asymptotic formula with no error 
term proven by Greenberg \cite{kn:gre83} and Villegas--Zagier \cite{kn:vilzag93}.  We can also give an upper bound that is
probably too large by only one logarithm for the second moment of
the $L$-functions in this family.  We conclude, by Cauchy's
inequality, that at least $N/(\log^2N)$  of the first $N$
$L$-functions in this family do not vanish at their central point.
 Moreover, if we assume that the Riemann Hypothesis holds for this family,
then we can compute the one-level density for this family with a restricted class of test functions, from which it follows
 that at least $1/4$ of the $L$-functions in this
family do not vanish at their central point.

The family of $L$-functions we consider are associated with a
sequence of Hecke Gr{\"o}ssencharacters. To make things concrete,
we deal with one specific case, the Gr{\"o}ssencharacters
associated with the field $Q(\sqrt{-7})$, following the paper by
Gross and Zagier \cite{kn:grozag80}.  Many families of
$L$-functions have been studied with a view to determining whether
they show the unitary, orthogonal or symplectic symmetry type of
random matrix theory, as proposed by Katz and Sarnak
\cite{kn:katzsarnak99a}. We investigate the moments at the central
point and the one-level density of the zeros of the Hecke
$L$-functions and find that these agree with the hypothesis that
the symmetry type of this family is {\it orthogonal}.

Some of this research was carried out during the MSRI program {\it Arithmetic Statistics} in 2011. The authors would like to thank 
MSRI for the hospitable working environment. The authors also thank the referee for very helpful comments 
and  for bringing to the attention of the authors some of the literature about this problem that they had 
missed in an earlier version of the paper.

\subsection{Background to Hecke $L$-functions}

The book of Iwaniec and Kowalski \cite{kn:iwakow}, Section 3.8, is
a good reference for the material in this section, as is \cite{kn:grozag80}, from which much of this material is taken.  See also \cite{kn:vilzag93}.  It should be noted that we use the analytic normalisation which places the central point at 1/2.   The integers
of ${\mathbb Q}(\sqrt{-7})$ are all numbers of the form $a+b\eta $
where $a$ and $b$ are integers and
\begin{equation}\eta=\frac{1+\sqrt{-7}}{2}.
\end{equation} The norm of $a+b\eta $ is
\begin{equation}
N(a+b\eta )=(a+b\eta )(a+b\overline{\eta})=a^2+ab
+2b^2.
\end{equation}
The field ${\mathbb Q}(\sqrt{-7})$ has class number 1 so that the
ideals are generated by the integers $a+b\eta$. The only units are
$+1$ and $-1$, so that each ideal has two generators.

The Dedekind zeta-function of the field $K={\mathbb Q}(\sqrt{-7})$
is
\begin{equation}
\zeta_K(s)=\frac 12 \sum_{(a,b)\ne
(0,0)}\frac{1}{N(a+b\eta )^s}=\zeta(s)L(s,\chi_{-7})
\end{equation}
where $\chi_{-7}(n)=\left(\frac n 7\right)$ is  the Legendre
symbol.

Note that
\begin{equation}\frac 12 \sum_{(a,b)\ne
(0,0)}q^{a^2+ab+2b^2}=1+\sum_{n=1}^\infty a_n q^n=1+q+2 q^2+3
q^4+q^7+4 q^8+2 q^{11}+2 q^{14}+\dots
\end{equation}
where
\begin{equation}\sum_{n=1}^\infty
\frac{a_n}{n^s}=\zeta_K(s)=(1+\frac{1}{2^s}+\frac{1}{3^s}+\dots)(1+\frac{1}{2^s}-\frac{1}{3^s}+
\frac{1}{4^s}-\frac{1}{5^s}-\frac{1}{6^s}+\frac{1}{8^s}+\dots).
\end{equation}

The Hecke character we are interested in is defined by
\begin{equation}
\chi((a+b\eta ))=\epsilon_{a,b} (a+b\eta )
\end{equation}
provided that $a+b\eta $ is relatively prime to
$\sqrt{-7}$  (otherwise the value of $\chi$ is 0). Here the choice
of $\epsilon=\pm 1$ is determined by
\begin{equation}
(a+b\eta )^3 \equiv \epsilon_{a,b}  \bmod \sqrt{-7}.
\end{equation}
  This amounts
to whether $a^3-2 a^2 b-a b^2+b^3$ is congruent to $\pm 1$ modulo
7. Thus,
\begin{equation}\epsilon_{a,b}=\left(\frac {a^3-2 a^2 b-a
b^2+b^3}{7}\right) .
\end{equation}

The Hecke $L$-function is
\begin{equation}L(s,\chi)=\frac 12
\sum_{(a,b)\ne (0,0)}\frac{\chi((a+b\eta ))}{(a^2+a b+2
b^2)^{s+1/2}}
\end{equation}
 which can be more simply written as
\begin{equation}L(s,\chi)=\frac 12 \sum_{(a,b)\ne (0,0)}\frac{ (a+b
\eta )\left(\frac{a^3-2 a^2 b-a b^2+b^3}7\right)}{(a^2 +
a b +2 b^2)^{s+1/2}}.
\end{equation}

This is the $L$-function of a cusp form of level 49 and weight $2$
and is the $L$-function of the elliptic curve $y^2+xy=x^3-x^2-2
x-1$, a rank 0 CM
 elliptic curve of conductor 49. The $L$ function $L_E(s)=L(s,\chi)$
satisfies the functional equation
\begin{equation}
\left(\frac{7}{2\pi}\right)^s\Gamma(s+1/2)L(s,\chi)=\Phi(s)=\Phi(1-s).
\end{equation}

We are interested in the primitive parts of the $L$-functions of
the symmetric powers of $L(s,\chi)$.  This amounts to looking at a
sequence of Hecke Gr{\"o}ssencharacters, denoted by
$\chi^{2n-1}$, $n=1,2,\ldots$. The series for $L(s,\chi^{2n-1})$
is
\begin{equation}L(s,\chi^{2n-1})=\frac 12
\sum_{(a,b)\ne (0,0)}\frac{ (a+b
\eta )^{2n-1}\left(\frac{a^3-2 a^2 b-a b^2+b^3}
7\right)}{(a^2 + a b +2 b^2)^{s+n-1/2}}.
\end{equation}
(Note that $L(s,\chi^{2n})$ is identically zero.) The Euler product for  $L(s, \chi^{2n-1})$ is 
\begin{equation}\label{eq:factoring} L(s, \chi^{2n-1})=\prod_{p=a^2+ab+2 b^2}\left(1-\frac{\epsilon_{a,b}(a+b \eta)^{2n-1}}{p^{s+n+1/2}}\right)^{-1}
\left(1-\frac{\epsilon_{a,b}(a+b \overline{\eta})^{2n-1}}{p^{s+n+1/2}}\right)^{-1}.
\end{equation}
In general, if 
\begin{equation}L(s)=\prod_p\left(1-\frac{\alpha_p}{p^s}\right)^{-1}\left(1-\frac{\overline{\alpha_p}}{p^s}\right)^{-1}\end{equation}
with $|\alpha_p|=1$, then the symmetric $k$th power is (up to some bad factors)
\begin{equation}L(s,\operatorname{sym}^k)=\prod_p\left(1-\frac{\alpha_p^k}{p^s}\right)^{-1}\left(1-\frac{\alpha_p^{k-2}}{p^s}\right)^{-1}
\dots  \left(1-\frac{\overline{\alpha_p^{k-2}}}{p^s}\right)^{-1}\left(1-\frac{\overline{\alpha_p^{k}}}{p^s}\right)^{-1}.\end{equation}
Thus we see in our situation for the symmetric powers of the L-function of a CM elliptic curve that  
\begin{equation}L(s,\chi,\operatorname{sym}^{2n-1})=L(s,\chi^{2n-1}) L(s,\chi^{2n-3})L(s,\chi^{2n-5})\dots L(s,\chi).\end{equation}
See \cite{kn:ragsha08} section 4 for more explicit details.

 It is convenient to define the function $\chi^{(2n-1)}$ at
positive rational integers $m$ by
\begin{equation}\chi^{(2n-1)}(m)
= \frac 1 2 \sum_{a^2+ab +2b^2=m}
\chi^{2n-1}((a+b\eta)).\end{equation} Then
\begin{equation}L(s,\chi^{2n-1})=\sum_{m=1}^\infty \frac{\chi^{(2n-1)}(m)}{m^{s+n-1/2}}
\end{equation}
 The functional
equation for $L(s,\chi^{2n-1})$ is
\begin{equation}
\left(\frac{7}{2\pi}\right)^s\Gamma(s+n-1/2)L(s,\chi^{2n-1})=
\Phi_{2n-1}(s)=(-1)^{n-1}\Phi_{2n-1}(1-s)
\end{equation}
and in asymmetric form
 \begin{equation}
 L(s,\chi^{2n-1})=(-1)^{n-1}X_{2n-1}(s)L(1-s,\chi^{2n-1})
\end{equation}
where
\begin{equation}X_{2n-1}(s)=\left(\frac{7}{2\pi}\right)^{1-2s}\frac{\Gamma(1-s+n-1/2)}{\Gamma(s+n-1/2)}.
\end{equation}
 Here the center of the
critical strip is at $s=1/2$. 

Using  Hecke's standard method, if $L(s)=\sum_{m=1}^\infty a_m m^{-s}$ is entire then the functional equation
\begin{equation} Q^s \Gamma(s+a) L(s)=Q^{1-s}\Gamma(1-s+a) L(1-s)\end{equation}
is equivalent (via Mellin transforms) to 
\begin{equation}f(y):= \sum_{m=1}^\infty m^a a_m e^{-my/Q} =y^{-2a-1}f(1/y).\end{equation}
Therefore,
\begin{eqnarray}
Q^s \Gamma(s+a)L(s)=Q^{-a}\int_0^\infty f(y) y^{s+a}\frac{dy}{y}&=&Q^{-a}\int_1^\infty \left(f(y) y^{s+a} +f(1/y) y^{-s-a}\right)\frac{dy}{y}
\end{eqnarray}
whence
\begin{eqnarray}
L(1/2)&=&\frac{2Q^{-a-1/2}}{\Gamma(1/2+a)} \int_1^\infty f(y) y^{a+1/2} \frac{dy}{y}\\
&=& \frac{2}{\Gamma(1/2+a)}\sum_{m=1}^\infty \frac{a_m}{m^{1/2}} \int_{m/Q}^\infty e^{-y} y^{a+1/2} \frac{dy}{y}.\nonumber
\end{eqnarray}
We apply this formula with $a=n-1/2$ and $Q=\frac{7}{2\pi}$ and use the formula for the incomplete Gamma function:
\begin{equation}
\Gamma(b,z)=\int_z^\infty y^{b} e^{-y}~\frac{dy}{y}.
\end{equation}
 In this way, if $n$
is odd, we obtain
\begin{equation}L(1/2,\chi^{2 n-1})=\frac{2 }{(n-1)!}
\sum_{m=1}^\infty \frac{\chi^{(2n-1)}(m)}{m^n}\Gamma(n,\frac{2\pi
m}{7}) .
\end{equation}

\section{Moments of the $L$-function at the central point}

One way to test the symmetry type of a family of $L$-functions is
to compute average values of the $L$-functions evaluated at the
central point.  In families displaying orthogonal symmetry the
average of the  $k$th power of the central value of the
$L$-function grows like the $k(k-1)/2$ power of the asymptotic
variable, where as for unitary symmetry the growth is like the
$k^2$ power and symplectic symmetry shows $k(k+1)/2$ power growth.
In the case of our family of Hecke characters, the asymptotic
parameter is $\log N$, so to agree with the predictions of
orthogonal symmetry we expect the first moment to be
asymptotically constant and the second moment to grow like $\log
N$.

\subsection{The first moment}

We wish to compute
\begin{eqnarray}
\mathcal M_r(N):= \frac{1}{N}\sum_{n=1 }^N   L(1/2,\chi^{4n-3})^r
\end{eqnarray}
asymptotically when $r=1$ and to give an upper bound when $r=2$.

We note that Greenberg (see \cite{kn:gre83}, page 258) states that such  an asymptotic formula for this first moment
(with no explicit error term) would follow from a formula of his, provided it were known that  $L(1/2,\chi^{4n-3})\ge 0$. 
Villegas and Zagier \cite{kn:vilzag93} prove this non-negativity. It is instructive to give a direct treatment from 
first principles. In addition we have an explicit error term. 

\begin{theorem}\label{theo:mom1}
As $N\to \infty$, we have
\begin{eqnarray}
\mathcal M_1(N)= \frac{1}{N}\sum_{n=1 }^N
L(1/2,\chi^{4n-3})=\frac{2\pi }{\sqrt{7}}+O\left(\frac{\log
N}{\sqrt{N}}\right).
\end{eqnarray}
\end{theorem}
\begin{proof}
We have
\begin{eqnarray}
\mathcal M_1(N)=\frac{2}{N}\sum_{m=1}^\infty
\frac{1}{\sqrt{m}}\sum_{n=1}^N \frac{\Gamma(2n-1,\frac{2\pi m}{7})
}{\Gamma(2n-1)} \frac{\chi^{(4n-3)}(m)}{m^{2n-3/2}}.
\end{eqnarray}
Now
\begin{eqnarray}
\chi^{(4n-3)}(m)= \frac{1}{2}\sum_{a^2+ab+2b^2=m}
\left(\frac{a^3-2a^2b-ab^2+b^3}{7}\right)
(a+b\eta )^{4n-3}
\end{eqnarray}
  so that
\begin{eqnarray}
\mathcal M_1(N)=\frac{1}{N}\sum_{(a,b)\ne (0,0)}
\frac{\left(\frac{a^3-2a^2b-ab^2+b^3}{7}\right)}{\sqrt{a^2+ab+2b^2}}\sum_{n=1}^N
\frac{\Gamma(2n-1,\frac{2\pi (a^2+ab+2b^2)}{7}) }{\Gamma(2n-1)}
 \delta_{a,b}^{4n-3}
\end{eqnarray}
where
\begin{eqnarray}
\delta_{a,b}=\frac{a+b\eta }{\sqrt{a^2+ab+2b^2}}=e(\theta_{a,b}),
\end{eqnarray}
say, where $e(x)=\exp(2\pi i x)$.  Now suppose that $f(x)$ is a positive, smooth,  increasing
function on $[0,\infty)$. Then for any real number $\theta$,
\begin{eqnarray}
\sum_{n=1}^N f(n) e ( n \theta)=\int_{1^-} ^ N f(u) d \Sigma(u)
=f(N)\Sigma(N) -\int_1^N f'(u)\Sigma(u)~du
\end{eqnarray}
where
\begin{eqnarray} \Sigma(u)=\sum_{n\le u} e(n \theta)=\frac{e^{2\pi i (1+[u])\theta}-e^{2\pi i \theta}}{1-e^{2\pi i \theta}}
=\frac{e^{2\pi i (1/2+[u])\theta}-e^{\pi i \theta}}{e^{-\pi i
\theta}-e^{\pi i \theta}}.
\end{eqnarray}
It follows that
\begin{eqnarray}
|\Sigma(u)|\le \frac{1}{|\sin \pi \theta|} .
\end{eqnarray}
Thus,
\begin{eqnarray}
\sum_{n=1}^N f(n) e ( n \theta)\ll  \frac{|f(N)|}{|\sin \pi
\theta|}.
\end{eqnarray}
Now
\begin{eqnarray}
f(n)= \frac{\Gamma(2n-1,\frac{2\pi (a^2+ab+2b^2)}{7})
}{\Gamma(2n-1)}
\end{eqnarray}
does satisfy the properties described above, and
\begin{eqnarray}
 \frac{\Gamma(2n-1,x )}{\Gamma(2n-1)}\ll e^{-x/n}.
 \end{eqnarray}
 Thus,
 \begin{eqnarray}  &&
\mathcal M_1(N)=\frac{1}{N}\sum_{(a,b)\ne (0,0)\atop
{4\theta_{a,b}\in \mathbb Z}}
\frac{\left(\frac{a^3-2a^2b-ab^2+b^3}{7}\right)}{\sqrt{a^2+ab+2b^2}}\sum_{n=1}^N
\frac{\Gamma(2n-1,\frac{2\pi (a^2+ab+2b^2)}{7}) }{\Gamma(2n-1)}
 \delta_{a,b}^{4n-3} \\
 &&\qquad \qquad \qquad +O\left(\frac{1}{N}\sum_{ 4\theta_{a,b}\notin \mathbb Z}
 \frac{e^{-(a^2+ab+2b^2)/N}}{\sqrt{a^2+ab +2b^2}} \frac{1}{|\sin 4 \pi \theta_{a,b}|}\right)
\end{eqnarray}
Now
\begin{eqnarray}
\delta_{a,b}=\cos 2\pi \theta_{a,b}+i\sin
2\pi\theta_{a,b}=\frac{a+b\eta }{\sqrt{a^2+ab+2b^2}}=e(\theta_{a,b})
\end{eqnarray}
so that
\begin{eqnarray}
\sin 4\pi \theta_{a,b}=2\sin 2\pi \theta_{a,b}\cos 2\pi
\theta_{a,b}=\frac{(a+b/2)b\sqrt{7}}{a^2+ab +2b^2}.
\end{eqnarray}
If $4\theta_{a,b}\notin \mathbb Z$ then
\begin{eqnarray}\label{eqn:sin}
\frac{1}{|\sin 4\pi \theta_{a,b}|} \ll \frac{a^2+ab+2
b^2}{|a+b/2||b|}\ll \sqrt{a^2+ab+2b^2}\max\{1/|a+b/2|,1/|b|\}
\end{eqnarray}
  since $a^2+ab +2b^2=(a+b/2)^2+7b^2/4.$  Thus, the $O$-term above is
  \begin{eqnarray}&
  \ll& \frac{1}{N}\sum_{ {b\ne 0}\atop{ 2a\ne -b}}
 e^{-(a^2+ab+2b^2)/N} \max\left\{\frac 1{|a+b/2|},\frac 1{|b|}\right\}\nonumber\\
 &\ll& \frac{1}{N}\sum_{ {b\ne 0}\atop{ 2a\ne -b}}
 e^{-(a^2+ab+2b^2)/N} \left(\frac 1{|a+b/2|}+\frac 1{|b|}\right)\nonumber\\
 &\ll& N^{-1/2}\log N.
\end{eqnarray}

If $4\theta_{a,b}\in \mathbb Z$ then either $b=0$ or $a=-b/2$ and
$\delta_{a,b}^4=1.$ In the situation that $a=-b/2$ we have
$\delta_{a,b}=a\sqrt{-7}$ is not coprime with $\sqrt{-7}$. So
these terms do not contribute anything. Thus, we have
 \begin{eqnarray}
\mathcal M_1(N)&=&\frac{1}{N}\sum_{{b=0}\atop {a\ne 0}}
\frac{\left(\frac{a^3-2a^2b-ab^2+b^3}{7}\right)}{\sqrt{a^2+ab+2b^2}}\delta_{a,b}^{-3}\sum_{n=1}^N
\frac{\Gamma(2n-1,\frac{2\pi (a^2+ab+2b^2)}{7})
}{\Gamma(2n-1)}\\&&\qquad \nonumber +O(N^{-1/2}\log N) .
\end{eqnarray}

\begin{lemma} (Tricomi \cite{kn:tricomi50})\label{lem:tricomi}
Suppose that $b>0$ and $n>0$. Let
\begin{eqnarray}
\gamma(n,b)= \int_0^b e^{-x}x^n\frac{dx}{x} .
\end{eqnarray}
Then, as $n\to \infty$,
\begin{eqnarray}
\frac{\gamma(n+1,n-y\sqrt{2n})}{\Gamma(n+1)}
=\frac{1}{2}\operatorname{erfc}(y)-\frac{\sqrt{2}}{3\sqrt{\pi n}}
(1+y^2)e^{-y^2} +O\left(\frac 1 n\right).
 \end{eqnarray}
 Here $y$ is any real number and
 \begin{eqnarray}
 \operatorname{erfc}(y)=\frac{2}{\sqrt{\pi}}\int_y^\infty e^{-t^2} ~dt .
 \end{eqnarray}
 For $y>0$,
 \begin{eqnarray}
 \operatorname{erfc}(y)= \frac{e^{-y^2}}{\sqrt{\pi}y}\left(1-\frac{1}{2y^2}+\frac{3}{4y^4}+\dots \right)
 \end{eqnarray}
 while
 \begin{eqnarray}
 \operatorname{erfc}(-y)=1-\operatorname{erfc}(y).
 \end{eqnarray}

 \end{lemma}

 We have, recalling $\left(\frac{-1}{7}\right)=-1$ and
 $\delta_{a,0}=a/|a|$,
 \begin{eqnarray}
\mathcal M_1(N)&=&\frac{2}{N}\sum_{a=1}^\infty \frac{\left(\frac a
7\right)}{a}\sum_{n=1}^N  \frac{\Gamma(2n-1,\frac{2\pi a^2 }{7})
}{\Gamma(2n-1)} +O(N^{-1/2}\log N).
\end{eqnarray}
We split the sum into four pieces:
$(\Sigma_1+\Sigma_2+\Sigma_3+\Sigma_4)/N$ where
\begin{eqnarray}
\Sigma_1= 2\sum_{n=1}^N \sum_{a^2\le
C_1n-C_2\sqrt{n}}\frac{\left(\frac a 7\right)}{a}
\frac{\Gamma(2n-1,\frac{2\pi a^2 }{7}) }{\Gamma(2n-1)}
\end{eqnarray}
\begin{eqnarray}
\Sigma_2= 2\sum_{n=1}^N \sum_{C_1n-C_2\sqrt{n}<a^2\le
C_1 n+C_2\sqrt{n}}\frac{\left(\frac a 7\right)}{a}
\frac{\Gamma(2n-1,\frac{2\pi a^2 }{7}) }{\Gamma(2n-1)}
\end{eqnarray}
\begin{eqnarray}
\Sigma_3= 2\sum_{n=1}^N \sum_{C_1n+C_2\sqrt{n}<a^2\le
2n}\frac{\left(\frac a 7\right)}{a}  \frac{\Gamma(2n-1,\frac{2\pi
a^2 }{7}) }{\Gamma(2n-1)}
\end{eqnarray}
and
\begin{eqnarray}
\Sigma_4= 2\sum_{n=1}^N \sum_{ a^2> 2n}\frac{\left(\frac a
7\right)}{a}  \frac{\Gamma(2n-1,\frac{2\pi a^2 }{7})
}{\Gamma(2n-1)};
\end{eqnarray}
here $C_1=\frac{7}{\pi}$ and $C_2$ is a large constant. $C_1$ is chosen so that the two arguments in the incomplete gamma function are approximately equal; see Lemma \ref{lem:tricomi} which describes this transition range. We have
\begin{eqnarray}
\Sigma_1=2\sum_{n=1}^N \sum_{a^2\le
C_1n-C_2\sqrt{n}}\frac{\left(\frac a 7\right)}{a} -\Sigma_{1,a}
\end{eqnarray}
where
\begin{eqnarray}
\Sigma_{1,a}= 2\sum_{n=1}^N \sum_{a^2\le
C_1n-C_2\sqrt{n}}\frac{\left(\frac a
7\right)}{a}\frac{\gamma(2n-1,\frac{2\pi a^2 }{7})
}{\Gamma(2n-1)}.
\end{eqnarray}
Now
\begin{eqnarray}
\sum_{a^2\le C_1n-C_2\sqrt{n}}\frac{\left(\frac a
7\right)}{a}=\sum_{a=1}^\infty \frac{\left(\frac a
7\right)}{a}+O\left(\frac{1}{\sqrt{n}}\right)
\end{eqnarray}
so that
\begin{eqnarray}
\Sigma_1=2L(1,\chi_{-7})N+O(\sqrt{N})-\Sigma_{1,a}.
\end{eqnarray}
  Using the Lemma, we have
\begin{eqnarray}
\Sigma_{1,a}\ll \sum_{n=1}^N \sum_{a^2\le C_1 n-C_2\sqrt{n} }
a^{-1}(\operatorname{erfc}(y) +y^2e^{-y^2}n^{-1/2}+O(1/n))
\end{eqnarray}
where $
y=(2n-2-2\pi a^2/7)/\sqrt{4n-4}$.    We illustrate how
to estimate this sum with a simplified version that omits the
constants.
\begin{eqnarray}
\sum_{n=1}^N \sum_{a^2\le n-\sqrt{n}} \frac{1}{a}
\operatorname{erfc}\left(\frac{(n-a^2)}{\sqrt{n}}\right) &\ll&
\sum_{n=1}^N \sum_{a^2\le n-\sqrt{n}} a^{-1} \frac{e^{-\frac{(n-a^2)^2}{n}}}
{\frac{(n-a^2)}{\sqrt{n}}}\nonumber\\
&\ll& \sum_{n=1}^N \int_1^{ \sqrt{n-\sqrt{n}}} u^{-1}
\frac{e^{-\frac{(n-u^2)^2}{n}}}{\frac{(n-u^2)}{\sqrt{n}}}~du\nonumber\\
&=& \sum_{n=1}^N \int_{\sqrt{n}}^{ n-1}  \frac{e^{-\frac{v^2}{n}}}
{\frac{v}{\sqrt{n}}}\frac{dv}{n-v}\nonumber\\
&=& \sum_{n=1}^N \sqrt{n}\int_{1}^{ \sqrt{n}-\frac{1}{\sqrt{n}}}
\frac{e^{-t^2}}{t}\frac{dt}{n-t\sqrt{n}}.
\end{eqnarray}
Now split the integral into $1\le t \le \sqrt{n}/2$ and
$\sqrt{n}/2\le t \le \sqrt{n}-1/\sqrt{n}$ to see that  this sample
sum is $\ll \sum_{n=1}^N \frac{1}{\sqrt{n}}\ll \sqrt{N}$.  We can
treat the part with $y^2e^{-y^2}/\sqrt{n}$ in a similar way and
the $1/n$ part trivially. In this way we have
\begin{eqnarray} \Sigma_{1,a}\ll \sqrt{N}.
\end{eqnarray}
The inner sum over $a$ of $\Sigma_2$ has a bounded number of
terms, each of which is $\ll 1/a \ll 1/\sqrt{n}$. Thus
$\Sigma_2\ll \sqrt{N}$. We can treat $\Sigma_3$ exactly as we did
$\Sigma_{1,a}$. Finally, the estimation of $\Sigma_4$ is like the
estimation of
\begin{eqnarray}
 \sum_{n=1}^N \sum _{a^2>2n} a^{-1}\int_{a^2}^\infty e^{-t}t^{n-1}~dt /\Gamma(n)&\ll& \sum_{n=1}^N \int_{2n}^\infty e^{-t}t^{n-1}\sum_{a\le \sqrt{t}}
 \frac{1}{a} ~dt/\Gamma(n)\nonumber\\
 &\ll & \sum_{n=1}^Ne^{-2n} (2n)^{n-1} (\log 2n)/\Gamma(n)\ll\sum_{n=1}^N e^{-n}\ll1.
 \end{eqnarray}

Thus, we conclude that
\begin{eqnarray}
\mathcal M_1(N)&=&2L(1,\chi_{-7})+O\left(\frac{\log
N}{\sqrt{N}}\right)=\frac{2\pi }{\sqrt{7}}+O\left(\frac{\log
N}{\sqrt{N}}\right).
\end{eqnarray}

\end{proof}

\subsection{The second moment}
We can give an upper bound for the second moment $\mathcal M_2(N).$
\begin{theorem} We have
\begin{equation}
\mathcal M_2(N) \ll  \log^2 N.
\end{equation}
\end{theorem}
Note that this result implies the convexity bound
  \begin{eqnarray}
L(1/2,\chi^{4n-3}\ll n^{1/2}\log n,
\end{eqnarray}
addressing Greenberg's question on the size of these $L$-functions (see \cite{kn:gre83}, page 258).
\begin{proof}
The proof follows exactly along the lines above. The only thing extra that we need is a bound
for summing the inverse of a quadratic form.
\begin{lemma}\label{lem:quadraticform} Let $Q$ be a non-degenerate quadratic form in 4 variables with integer coefficients. Then
\begin{equation}
\sum_{a,b,A,B\le X\atop
Q(a,b,A,B)\ne 0}\frac{1}{|Q(a,b,A,B)|} \ll X^2 \log^2 X.
\end{equation}
\end{lemma}
For example, \begin{equation}\sum_{a,b,A,B\le X\atop a\ne b, A\ne
B} \frac{1}{|(a-b)(A-B)|}\ll X^2 \log^2 X.\end{equation} We leave
the proof as an interesting exercise for the reader.

 Here is a sketch of the proof of the theorem.  Take the square of the formula
for the central value and average over $n$:
 \begin{eqnarray}
\mathcal M_2(N)&=&\frac{1}{N}\sum_{(a,b)\ne (0,0)\atop (A,B)\ne (0,0)}
\frac{\left(\frac{a^3-2a^2b-ab^2+b^3}{7}\right)}{\sqrt{a^2+ab+2b^2}}
\frac{\left(\frac{A^3-2A^2B-AB^2+B^3}{7}\right)}{\sqrt{A^2+AB+2B^2}}\nonumber\\
&&\qquad \times \sum_{n=1}^N  \frac{\Gamma(2n-1,\frac{2\pi
(a^2+ab+2b^2)}{7}) }{\Gamma(2n-1)} \frac{\Gamma(2n-1,\frac{2\pi
(A^2+AB+2B^2)}{7}) }{\Gamma(2n-1)}
 (\delta_{a,b}\delta_{A,B})^{4n-3}.
\end{eqnarray}
The inner sum is a geometric series with a smooth weight, as in the proof of Theorem \ref{theo:mom1}.   Now,
with $\eta=(1+\sqrt{-7})/2$,
 \begin{eqnarray}
 \delta_{a,b}\delta_{A,B} &=& \frac{a+b\eta}{|a+b \eta|}
 \frac{A+B \eta}{|A+B\eta|}=\frac{(aA-2bB)+(Ab+aB+bB)\eta}{\sqrt{a^2+ab+2b^2}
 \sqrt{A^2+AB+2B^2}}\nonumber\\
 &=&\delta_{aA-2bB,Ab+aB+bB}=e(\theta_{aA-2bB,Ab+aB+bB}).
\end{eqnarray}
Just as above, we have
\begin{eqnarray}
\frac{1}{|\sin 4\pi \theta_{aA-2bB,Ab+aB+bB}|} \ll  &&
\sqrt{a^2+ab+2b^2}\sqrt{A^2+AB+2B^2}\nonumber
\\&&\quad \times \max\{1/|2aA-4bB+Ab+aB+bB|,1/|Ab+aB+bB|\}.
\end{eqnarray}
Thus, by Lemma \ref{lem:quadraticform}, the terms with  $4\pi
\theta_{aA-2bB,Ab+aB+bB} \notin \mathbb Z$ contribute an amount
which is $\ll N\log^2 N$. If $4\pi \theta_{aA-2bB,Ab+aB+bB} \in
\mathbb Z$, then it must be the case that either $Ab+aB+bB=0$ or
else $2aA-4bB+Ab+aB+bB=0$. As before, in the second case the
coefficient of this term is 0. Thus,  $Ab+aB+bB=0$. If
$(a,b)=1=(A,B)$, then we have $B(a+b)=-Ab$ and $b(A+B)=-aB$ so
that $b\mid B$ and $B\mid b$. If $B=0$, then $b=0$ and vice versa.
If $B=-b$ then $ A=a+b$ and if $B=b\ne 0$, then $A=-a-b$.
 In any of these events we have
 \begin{equation}\sum_{a,b,A,B\le X \atop Ab+aB+bB=0} \frac{1}{\sqrt{a^2+ab+2b^2}
 \sqrt{A^2+AB+2B^2}}\ll \log^2 X.
 \end{equation}
We conclude that, in this diagonal case, the sum over $n$  is $N$
and  the sum over $a,A,b,B$ is $\ll \log^2N$.  Thus, we have shown that  $\mathcal M_2(N)\ll \log^2 N$ as desired.
\end{proof}
\begin{corollary}
For at  least $N /(\log N)^2$ values of $n \le N$ we have  $L(2n-1,\chi^{4n-3})\ne 0.$
\end{corollary}
This follows from a standard use of Cauchy's inequality:
\begin{eqnarray}
|\sum_{n=1}^N L(1/2,\chi^{4n-3})|^2 \le \bigg(\sum_{n\le N\atop L(1/2,\chi^{4n-3})\ne 0} 1\bigg)\left(\sum_{n=1}^N |L(1/2,\chi^{4n-3})|^2\right)
\end{eqnarray}
whence \begin{equation} \sum_{n\le N\atop L(1/2,\chi^{4n-3})\ne
0}1\ge \frac{ (N \mathcal M_1(N))^2}{N\mathcal  M_2(N)}\gg \frac{N^2}{N
\log^2 N} \gg \frac{N}{\log^2N}.\end{equation}

\section{Moment conjectures}

In this section we use the moment conjectures described in
\cite{kn:cfkrs} to calculate the first and second moments of our
family of $L$-functions. The first moment agrees with the previous
section and both of the first two moments support the hypothesis
that the family has orthogonal symmetry.

\subsection{The first moment} We want to use the moment conjecture
recipe to find the first moment at the central point of the family
of $L$-functions $L(s,\chi^{4n-3})$ as we vary $n$. We want to
calculate
\begin{equation}
 \frac{1}{N}\sum_{n=1}^{N}
L(1/2+\alpha,\chi^{4n-3}).
\end{equation}

The functional equation for $L(s,\chi^{4n-3})$ looks like
\begin{equation}
\Big(\frac{7}{2\pi}\Big)^{s}\Gamma(s+2n-3/2)L(s,\chi^{4n-3})=\Phi_{4n-3}(s)=\Phi_{4n-3}(1-s),
\end{equation}
defining
\begin{equation}
{X}_{4n-3}(s):=\Big(\frac{7}{2\pi}\Big)^{1-2s}\frac{\Gamma(1-s+2n-3/2)}
{\Gamma(s+2n-3/2)},
\end{equation}
 and so
the main two terms in the approximate functional equation give us
\begin{eqnarray}
&&\frac{1}{N}\sum_{n=1}^{N} L(1/2+\alpha,\chi^{(4n-3)})\nonumber
\\
&&=\frac{1}{N}\sum_{n=1}^N\Bigg(\sum_m
\frac{\chi^{(4n-3)}(m)}{m^{2n-1+\alpha}}+{X}_{4n-3}(1/2+\alpha)\sum_m\frac{\chi^{(4n-3)}(m)}{m^{2n-1-\alpha}
}\Bigg).
\end{eqnarray}

The recipe instructs us to perform the average over $n$ over the
characters and the $X$ factor from the functional equation.  The
quantity we need to understand is
\begin{equation} \label{eqn:delt}
\delta(m)=\Bigg<\frac{\chi^{(4n-3)}(m)}{m^{2n-3/2}}\Bigg>=\lim_{N\rightarrow
\infty}\frac{1}{N}\sum_{n=1}^N
\frac{\chi^{(4n-3)}(m)}{m^{2n-3/2}}.
\end{equation}
We claim that $\delta$ is multiplicative.
To prove this claim it suffices to prove that
\begin{eqnarray*}
\lim_{N\to \infty}\frac{1}{N}\sum_{n=1}^N  \left(\frac{(a+b\eta)(c+d\eta)}{|a+b\eta||c+d\eta|}\right)^{4n}=
\lim_{N\to \infty}\frac{1}{N}\sum_{n=1}^N  \left(\frac{(a+b\eta) }{|a+b\eta| }\right)^{4n}
\lim_{N\to \infty}\frac{1}{N}\sum_{n=1}^N  \left(\frac{ (c+d\eta)}{ |c+d\eta|}\right)^{4n}
\end{eqnarray*}
whenever $(N(a+b\eta),N(c+d\eta))=1$. 
The only time that  
\begin{eqnarray*}
\lim_{N\to \infty}\frac{1}{N}\sum_{n=1}^N  \left(\frac{ (A+B\eta)}{ |A+B\eta|}\right)^{4n}
\end{eqnarray*}
is not zero is when 
\begin{equation}
 \left(\frac{ (A+B\eta)}{ |A+B\eta|}\right)^{4}=1
\end{equation}
or, equivalently, when $(A+B\eta)^4$ is real and positive, i.e. $A+B \eta$ is either real or purely imaginary, which translates to 
either $B=0$ or $A=-B/2$. 
Thus, under the assumption that  $(N(a+b\eta),N(c+d\eta))=1$, it suffices to prove  the assertion that  
$(a+b\eta) (c+d\eta)$ is either real or purely imaginary  if and only if each of  $(a+b\eta)$ and $(c+d\eta)$ is either real or purely imaginary.
On direction is clear. In the other direction, we
 consider two cases. In the first case suppose that
 $(a+b\eta)(c+d\eta)=A+B\eta$ with $B=0$.  Then 
\begin{equation} \label{eqn:gcd}
a'd+c'b=0
\end{equation}
 where $a'=2a+b$ and $c'=2c+d$ (note that $a',b,c',d$ are all integers).
Write $a'=gA'$ and $b=gB$ with $g=(a',b)$. Then equation (\ref{eqn:gcd}) becomes
$A'd+c'B=0$. Since $(A',B)=1$ it must be the case that $B\mid d$ and $A'\mid c'$, say $d=tB$ and $c'=-tA'$. Substituting back we have
\begin{eqnarray}
a=\frac{g(A'-B)}{2}\qquad b=gB\qquad c= \frac{-t(A'+B)}{2}\qquad d=Bt.
\end{eqnarray}
Then
\begin{eqnarray*}
N(a+b\eta)=a^2+ab+2b^2=(A'^2+7 B^2)\frac{g^2}{4} \qquad N(c+d\eta)=c^2+cd +2 d^2=(A'^2+7 B^2)\frac{t^2}{4}.
\end{eqnarray*}
If $B\ne 0$ then clearly $(N(a+b\eta),N(c+d\eta))>1$. Therefore $B=0$ which implies that $b=d=0$.
 Similar sorts of arguments work in the case that $(a+b\eta)(c+d\eta)=A+B\eta$ with $A=-B/2$. 
Thus, $\delta$ is multiplicative. Now we need to evaluate it at prime power arguments.
 It is not hard to show that $\delta$  vanishes at non-square arguments and that
\begin{equation}
\delta(p^2) = \left\{ \begin{array}{ccc} 0& {\rm if}& p=7\\ +1&
{\rm if} &(\tfrac{p}{7})=1\\-1&{\rm
if}&(\tfrac{p}{7})=-1\end{array}\right.
\end{equation}
and that
\begin{equation}
\delta(p^{2k})=\delta(p^2)^k.
\end{equation}
That is,
\begin{equation}
\delta(m)=\left\{ \begin{array}{ccc} 0& {\rm if}&m\neq
\square\\(\tfrac{\sqrt{m}}{7})& {\rm
if}&m=\square\end{array}\right..
\end{equation}
So, by the recipe we have
\begin{eqnarray}
&&\frac{1}{N}\sum_{n=1}^{N}
L(1/2+\alpha,\chi^{4n-3})\nonumber\\&&
\approx\frac{1}{N}\sum_{n=1}^N
\Bigg(\sum_m\frac{\delta(m)}{m^{1/2+\alpha}}+\Big<{X}_{4n-3}(1/2+\alpha)\Big>\sum_m
\frac{\delta(m)}{m^{1/2-\alpha}}\Bigg)\nonumber \\
&& =\frac{1}{N}\sum_{n=1}^N
\Bigg(\sum_m\frac{\delta(m^2)}{m^{1+2\alpha}}
+\Big<{X}_{4n-3}(1/2+\alpha)\Big>
\sum_m \frac{\delta(m^2)}{m^{1-2\alpha}}\Bigg)\nonumber \\
&& = \frac{1}{N}\sum_{n=1}^N \Bigg(\sum_m
\frac{(\tfrac{m}{7})}{m^{1+2\alpha}} +\Big<{X}_{4n-3}(1/2+\alpha)\Big>
\sum_m\frac{(\tfrac{m}{7})}{m^{1-2\alpha}}\Bigg)
\end{eqnarray}
So, the moment conjecture in this case would be
\begin{eqnarray}
&&\frac{1}{N}\sum_{n=1}^{N} L(1/2+\alpha,\chi^{4n-3})
=L(1+2\alpha,
\chi_{-7})\nonumber \\
&&\qquad\qquad+\Big(\frac{7}{2\pi}\Big)^{-2\alpha}
\frac{1}{N}\sum_{n=1}^N
\frac{\Gamma(2n-1-\alpha)}{\Gamma(2n-1+\alpha)}\;L(1-2\alpha,\chi_{-7})+O(N^{-1/2+\epsilon}).
\end{eqnarray}
  When
$\alpha=0$,
\begin{equation}
\frac{1}{N}\sum_{n=1}^NL(\tfrac{1}{2},\chi^{4n-3})=2L(1,\chi_{-7})
+O(N^{-1/2+\epsilon}).
\end{equation}
  Thus we have
  \begin{conjecture} Using the moment conjecture recipe from
  \cite{kn:cfkrs}
\begin{equation}
\frac{1}{N}\sum_{n=1}^NL(\tfrac{1}{2},\chi^{4n-3})=\frac{2\pi}{\sqrt{7}}
+O(N^{-1/2+\epsilon}).
\end{equation}
\end{conjecture}

\subsection{The second moment}   Now we calculate the second moment
using the moment conjecture:
\begin{eqnarray}
&&\sum_{n=1}^N
L(\tfrac{1}{2}+\alpha,\chi^{4n-3})L(\tfrac{1}{2}
+\beta,\chi^{4n-3}) \nonumber \\
&&\approx \sum_{n=1}^N\Bigg(\sum_{\ell}
\frac{\chi^{(4n-3)}(\ell)}{\ell^{2n-1+\alpha}}
+X_{4n-3}(1/2+\alpha)\sum_{\ell}\frac{\chi^{(4n-3)}(\ell)}
{\ell^{2n-1-\alpha}}\Bigg) \nonumber\\  \nonumber &&\qquad \qquad
\qquad \times \Bigg(\sum_{m}
\frac{\chi^{(4n-3)}(m)}{m^{2n-1+\beta}}
+{X}_{4n-3}(1/2+\beta)\sum_{m}\frac{\chi^{(4n-3)}(m)}
{m^{2n-1-\beta}}\Bigg)\nonumber \\
&&\approx \sum_{n=1}^N \Bigg(\sum_{\ell,m} \frac{\delta(\ell,m)}
{\ell^{1/2+\alpha}m^{1/2+\beta}}+\Big<{X}_{4n-3}(1/2+\alpha)\Big> \sum_{\ell,
m}\frac{\delta(\ell,m)} {\ell^{1/2-\alpha}m^{1/2+\beta}}\nonumber\\
&& \qquad\qquad \qquad + \Big<{X}_{4n-3}(1/2+\beta)\Big>
\sum_{\ell, m}\frac{\delta(\ell,m)}
{\ell^{1/2+\alpha}m^{1/2-\beta}}\\
 && \qquad\qquad \qquad+\Big<{X}_{4n-3}(1/2+\alpha){X}_{4n-3}(1/2+\beta)\Big>
\sum_{\ell,m}\frac{\delta(\ell,m)}{\ell^{1/2-\alpha}m^{1/2-\beta}}\Bigg).\nonumber
\end{eqnarray}
Here we define
\begin{equation}
\delta(\ell,m)=\lim_{N\rightarrow \infty}\frac{1}{N}\sum_{n=1}^{N}
\frac{\chi^{(4n-3)}(\ell)}{\ell^{2n-3/2}}\frac{\chi^{(4n-3)}(m)}{m^{2n-3/2}}.
\end{equation}
  We find that $\delta$ is multiplicative.  That is,
if $(\ell_1\ell_2, m_1 m_2)=1$, then
\begin{equation}\label{eqn:deltamultiplicative}
\delta(\ell_1 m_1,\ell_2 m_2)=
\delta(\ell_1,\ell_2) \delta(m_1,m_2).
\end{equation}
We leave the proof to the reader; basically it is an elaboration of the proof for the one-variable $\delta$ given immediately after
(\ref{eqn:delt}). 
The behaviour of $\delta$ is summarised as (for $a<b$)
\begin{equation}
\delta(p^a,p^b)=\left\{ \begin{array}{ccc} 0& {\rm if} & a+b {\rm
\; odd}, p=1,2,4 \mod 7\\ a+1 & {\rm if} & a+b {\rm \; even },
p=1,2,4 \mod 7\\ 0& {\rm if} & a {\rm \;or\;} b {\rm
\;odd}, p=3,5,6 \mod 7\\(-1)^{(a+b)/2} & {\rm if} & a {\rm
\;and \;} b {\rm \; even}, p=3,5,6 \mod
7\end{array}\right.
\end{equation}
The practical use of being multiplicative is that we can write the sum over $\delta$ as an Euler product:
\begin{equation}\label{eq:euler}
\sum_{\ell,m}\frac{\delta(\ell,m)}{\ell^{1/2+\alpha}m^{1/2+\beta}}
= \prod_p \sum_{a,b}
\frac{\delta(p^a,p^b)}{p^{(1/2+\alpha)a+(1/2+\beta)b}}.
\end{equation}
 For primes
$p=1,2,4\mod 7$,
\begin{eqnarray}\label{eq:primes1}
&&\sum_{a,b=0}^{\infty} \delta(p^a,p^b)
p^{-(\alpha+1/2)a}p^{-(\beta+1/2)b}  = \sum_{a=0}^{\infty}
\sum_{{b=0}\atop{a+b \;even}}^a
(b+1)p^{-(\alpha+1/2)a}p^{-(\beta+1/2)b} \nonumber
\\
&&+\sum_{b=0}^{\infty} \sum_{{a=0}\atop{a+b \;even}}^b
(a+1)p^{-(\alpha+1/2)a}p^{-(\beta+1/2)b}- \sum_{a=0}^{\infty}
(a+1)p^{-(\alpha+1/2)a}p^{-(\beta+1/2)b} \nonumber \\
&& = \frac{\Big(1-\frac{1}{p^{1+2\alpha}}\Big)^{-1} \Big(
1-\frac{1}{p^{1+\alpha+\beta}}\Big)^{-1}\Big(1-\frac{1}{p^{1+2\beta}}\Big)^{-1}}
{\Big(1+\frac{1}{p^{1+\alpha+\beta}}\Big)^{-1}}.
\end{eqnarray}
Note that if all the primes contributed an expression of this
form, the Euler product would yield
\begin{equation}
\frac{\zeta(1+2\alpha)\zeta(1+\alpha+\beta)^2\zeta(1+2\beta)}
{\zeta(2+2\alpha+2\beta)}
\end{equation}
which has a fourth order pole as $\alpha$ and $\beta$ approach
zero.

For primes $p=3,5,6 \mod 7$,
\begin{eqnarray}\label{eq:primes2}
&&\sum_{a,b=0}^{\infty} \delta(p^a,p^b)
p^{-(\alpha+1/2)a}p^{-(\beta+1/2)b}  = \sum_{a,b=0}^{\infty}
(-1)^{a+b} p^{-2a(\alpha+1/2)}p^{-2b(\beta+1/2)} \nonumber \\
&&\qquad =
\Big(1+\frac{1}{p^{1+2\alpha}}\Big)^{-1}\Big(1+\frac{1}{p^{1+2\beta}}\Big)^{-1}
=
\frac{\Big(1-\frac{1}{p^{2+4\alpha}}\Big)^{-1}\Big(1-\frac{1}{p^{2+4\beta}}\Big)^{-1}}
{\Big(1-\frac{1}{p^{1+2\alpha}}\Big)^{-1}\Big(1-\frac{1}{p^{1+2\beta}}\Big)^{-1}}
\end{eqnarray}
If all the primes had this contribution the resulting product
would be
\begin{equation}
\frac{\zeta(2+4\alpha)\zeta(2+4\beta)}
{\zeta(1+2\alpha)\zeta(1+2\beta)}.
\end{equation}
which has a second order zero. Thus, half the primes lead to a 4th order pole and the other half to a second order zero;
taken together over all the primes we have  a pole of order $\frac{4}{2} -\frac{2}{2}=1$, 
  which is what we expect from
an orthogonal family second moment.

Combining (\ref{eq:primes1}) and (\ref{eq:primes2}) we have that
(\ref{eq:euler}) is equal to
\begin{eqnarray}
&&=\frac{L(1+2\alpha,\chi_{-7})L(1+2\beta,\chi_{-7})
\zeta_7(1+\alpha+\beta) L(1+\alpha+\beta,\chi_{-7})}
{\zeta_7(2+2\alpha+2\beta)},
\end{eqnarray}
where
\begin{equation}
\zeta_7(s)=\prod_{p\neq 7} \Big(1-\frac{1}{p^s}\Big)^{-1}.
\end{equation}

Thus the second moment would be
\begin{eqnarray}
&&\frac{1}{N}
\sum_{n=1}^NL(1/2+\alpha,\chi^{4n-3})L(1/2+\beta, \chi^{4n-3}) =
\frac{1}{N} \sum_{n=1}^N \Bigg( \zeta(1+\alpha+\beta)
F(\alpha,\beta) \nonumber \\
&&\qquad\qquad\qquad+
\Big(\frac{7}{2\pi}\Big)^{-2\alpha}\frac{\Gamma(2n-1-\alpha)}{\Gamma(2n-1+\alpha)}
\zeta(1-\alpha+\beta) F(-\alpha,\beta)  \nonumber \\
&&\qquad\qquad\qquad+\Big(\frac{7}{2\pi}
\Big)^{-2\beta}\frac{\Gamma(2n-1-\beta)}{\Gamma(2n-1+\beta)}
\zeta(1+\alpha-\beta)F(\alpha,-\beta) \nonumber \\
&&\qquad\qquad\qquad +\Big(\frac{7}{2\pi}\Big)^{-2\alpha-2\beta}
\frac{\Gamma(2n-1-\alpha)}{\Gamma(2n-1+\alpha)}\frac{\Gamma(2n-1-\beta)}{\Gamma(2n-1+\beta)}
\zeta(1-\alpha-\beta)F(-\alpha,-\beta)\Bigg)\\
\nonumber
&&\qquad\qquad\qquad+O(N^{-1/2+\epsilon}),\label{eq:shiftedsecond}
\end{eqnarray}
where
\begin{equation}
F(\alpha,\beta)=\frac{L(1+2\alpha,\chi_{-7})L(1+2\beta,
\chi_{-7})L(1+\alpha+\beta,\chi_{-7}) \Big
(1-\frac{1}{7^{1+\alpha+\beta}}\Big)} {\zeta(2+2\alpha+2\beta)
\Big(1-\frac{1}{7^{2+2\alpha+2\beta}}\Big)}.
\end{equation}

Now we want to set $\alpha, \beta \rightarrow 0$.  We group the
first and fourth term in the numerator of (\ref{eq:shiftedsecond}) and in the
second and third term send $\alpha \rightarrow -\alpha$ (allowable
because we are going to take the limit $\alpha\rightarrow 0$) and
then factor out the exponential and gamma factors
\begin{equation}\label{eq:factorout}
\Big(\frac{7}{2\pi}\Big)^{2\alpha}\frac{\Gamma(2n-1+\alpha)}{\Gamma(2n-1-\alpha)}
\end{equation}
to leave exactly the first and fourth terms again.  The expression
(\ref{eq:factorout}) tends to 1 as $\alpha,\beta \rightarrow 0$.
Thus we expect that the second moment evaluated at the central point is
\begin{eqnarray}
&&\mathcal M_2(N):=\lim_{\alpha,\beta\rightarrow 0} \frac{1}{N} \sum_{n=1}^N
L(1/2+\alpha,\chi^{4n-3})L(1/2+\beta, \chi^{4n-3})\nonumber \\
&& \qquad \qquad =\lim_{\alpha,\beta\rightarrow 0}  \frac{1}{N} \sum_{n=1}^N
\Bigg( 1+\Big(\frac{7}{2\pi}\Big)^{2\alpha}
\frac{\Gamma(2n-1+\alpha)}{\Gamma(2n-1-\alpha)} \Bigg) \Bigg(
\zeta(1+\alpha+\beta)F(\alpha,\beta)\nonumber \\
&&\qquad\qquad\qquad
+\Big(\frac{7}{2\pi}\Big)^{-2\alpha-2\beta}\frac{\Gamma(2n-1-\alpha)}{\Gamma(2n-1+\alpha)}
\frac{\Gamma(2n-1-\beta)}{\Gamma(2n-1+\beta)}
\zeta(1-\alpha-\beta)F(-\alpha,-\beta)\Bigg)\nonumber \\
&&\qquad\qquad\qquad+O(N^{-1/2+\epsilon})\nonumber \\
&&\qquad \qquad =4\Bigg(f_0\gamma+f_1-f_0\log\frac{2\pi}{7} +f_0 \frac{1}{N+1}
\sum_{n=1}^N \frac{\Gamma'(2n-1)}{\Gamma(2n-1)} \Bigg)+O(N^{-1/2+\epsilon}),
\end{eqnarray}
where we expand $F(a,b)$ around $a=0,b=0$ as
\begin{equation}
F(a,b)=f_0+f_1a+f_1b+\cdots
\end{equation}
with
\begin{equation}
f_0=\frac{L(1,\chi_{-7})^3}{\zeta(2)}\frac{7}{8}
=\Big(\frac{\pi}{\sqrt{7}}\Big)^3 \frac{6}{\pi^2} \frac{7}{8} =
\frac{3\pi}{4\sqrt{7}}
\end{equation}
and
\begin{equation}
\frac{\partial}{\partial \alpha}
F(\alpha,\beta)\Big|_{\alpha,\beta=0} = f_1 =
f_0\Big(3\frac{L'}{L}(1,\chi_{-7}) -2\frac{\zeta'}{\zeta}(2)
+\frac{\log 7}{8}\Big).
\end{equation}
So, the our conjecture is that
\begin{equation}\label{eq:M2}
\mathcal M_2(N)=\frac{3\pi}{\sqrt{7}} \Big(\gamma+3\frac{L'}{L}(1,\chi_{-7})
-2\frac{\zeta'}{\zeta}(2)+ \frac{\log 7}{8} -\log
\frac{2\pi}{7}+\frac{1}{N}\sum_{n=1}^N\frac{\Gamma'(2n-1)}{\Gamma(2n-1)}
\Big) +O(N^{-1/2+\epsilon}).
\end{equation}
It can be shown that this reduces to
\begin{conjecture}
\begin{eqnarray}
\mathcal M_2(N)&=&\frac{3\pi}{\sqrt{7}} (\log N +C) +O(N^{-1/2+\epsilon})
\end{eqnarray}
where
\begin{eqnarray}
C=
4\gamma  -3\log\frac{\Gamma(1/7)\Gamma(2/7)\Gamma(4/7)}
{\Gamma(3/7)\Gamma(5/7)\Gamma(6/7)}
-2\frac{\zeta'}{\zeta}(2)+ \frac{\log 7}{8} +\log
7\pi^2+3\log2-1
 .
\end{eqnarray}
\end{conjecture}

With $N=469$, computing $L$-values and evaluating  $\mathcal
M_2(N)$ numerically gives 28.37.  The main term of equation (\ref{eq:M2}) gives
28.35.

\section{One-level density}
In this section we assume the Riemann Hypothesis for the family of  $L(s,\chi^{4n-3})$
and calculate the one-level density for this family, valid for test functions whose Fourier transforms are supported in $[-\alpha,\alpha]$ where
$\alpha<1$. As a consequence of our calculation we can show
\begin{theorem} If the Riemann Hypothesis for the family $\{L(s,\chi^{4n-3})\}$ is true, then $L(1/2,\chi^{4n-3})\ne 0$ for at least $(1/4-\epsilon)N$ values of  $n\le N$.
  \end{theorem}

   Let
$\phi(t)=F(\tfrac{1}{2}+it)$ be an even test function whose
Fourier transform has compact support. Then $F$ is entire and
decays quickly with $|t|$.  We derive an explicit formula.
Defining
\begin{equation}\label{eq:defLambda}
-\frac{L'(s,\chi^{4n-3})}{L(s,\chi^{4n-3})}=\sum_{k=1}^\infty\frac{\Lambda_{4n-3}(k)}{k^s},
\end{equation}
we write
\begin{eqnarray}\label{eq:explicit1}
\frac{1}{2\pi i} \int_{(2)}
\frac{L'}{L}(s,\chi^{4n-3})F(s)ds &=& -\sum_{k=1}^\infty
\Lambda_{4n-3}(k) \frac{1}{2\pi i} \int_{(2)}
F(s)k^{-s}ds\nonumber \\
&=&-\sum_{k=1}^\infty \Lambda_{4n-3}(k) \frac{1}{2\pi \sqrt{k}}
\int_{-\infty}^{\infty}\phi(t) e^{-it\log k} dt
\nonumber \\
&=&-\frac{1}{2\pi} \sum_{k=1}^\infty \frac{\Lambda_{4n-3}(k)}{\sqrt{k}}
\hat{\phi}\left(\frac{\log k}{2\pi}\right).
\end{eqnarray}
Here we use the definition
\begin{equation}
\hat{f}(t)=\int_{-\infty}^{\infty} f(u)e(-ut)du.
\end{equation}

On the left side of (\ref{eq:explicit1}) we move the contour to
the vertical line with real part 1/4 (note we've assumed RH and so
have encircled all the zeros).  Thus we have, with
$\rho_n=1/2+\gamma_n$ a generic\footnote{The subscript $n$ on $\rho_n$ merely means that it is a zero of $L(s,\chi^{4n-3})$ and should not 
 be mistaken for the $n$th zero.} zero
\begin{eqnarray}
\frac{1}{2\pi i} \int_{(2)}
\frac{L'}{L}(s,\chi^{4n-3})F(s)ds &=&\sum_{\rho_n} F(\rho_n)
+\frac{1}{2\pi i} \int_{(1/4)}
\frac{L'}{L}(s,\chi^{4n-3})F(s)ds\nonumber \\
&=&\sum_{\gamma_n} \phi(\gamma_n) +\frac{1}{2\pi i} \int_{(1/4)}
\left( \frac{{X}'_{4n-3}}{{X}_{4n-3}}(s)-\frac{L'}{L}(1-s,\chi^{4n-3})\right) F(s) ds
\nonumber \\
&=& \sum_{\gamma_n} \phi(\gamma_n) +\frac{1}{2\pi}
\int_{-\infty}^{\infty}\frac{{X}'_{4n-3}}{X_{4n-3}}(1/2+it) \phi(t)dt\nonumber
\\
&&\qquad-\frac{1}{2\pi i} \int_{(3/4)}
\frac{L'}{L}(s,\chi^{4n-3})F(1-s)ds.
\end{eqnarray}
In the last integral the contour can be moved to the vertical line
with real part 2.
\begin{eqnarray}
-\frac{1}{2\pi i} \int_{(2)} \frac{L'}{L}(s,\chi^{4n-3})
F(1-s) ds &=& \sum_{k=1}^\infty \Lambda_{4n-3}(k)\;\frac{1}{2\pi i} \int_{(2)}
F(1-s) k^{-s}
ds\nonumber \\
&=&\sum_{k=1}^\infty \Lambda_{4n-3}(k)\;\frac{1}{2\pi i} \int_{(1/2)} F(s) k^{s-1}
ds\nonumber \\
&=&\sum_{k=1}^\infty \Lambda_{4n-3}(k)\;\frac{1}{2\pi } \int_{-\infty}^{\infty}
\frac{\phi(t)}{\sqrt{k}} k^{it}
dt\nonumber \\
&=&\frac{1}{2\pi}\sum_{k=1}^\infty\frac{\Lambda_{4n-3}(k)}{\sqrt{k}}
\hat{\phi}\left( \frac{\log k}{2\pi}\right).
\end{eqnarray}
So,
\begin{equation}
\sum_{\gamma} \phi(\gamma) = -\frac{1}{2\pi}
\int_{-\infty}^{\infty} \frac{X'}{X}(1/2+it) \phi(t)dt -
\frac{1}{\pi} \sum_{k=1}^\infty \frac{\Lambda_{4n-3}(k)}{\sqrt{k}}
\hat{\phi}\left( \frac{\log k}{2\pi}\right).
\end{equation}

We really want to work with scaled zeros in the explicit formula,
so now we define a set of zeros $\widetilde{\gamma}_n$ that have
average consecutive spacing of 1.  We need to know how many zeros
$L(s,\chi^{4n-3})$ has in the interval $0<t<T$, where $T$ is large
but bounded.  We calculate the change in argument of
\begin{equation}
\xi(s,\chi^{4n-3})=\left(\frac{7}{2\pi}\right)^s
\Gamma(s+2n-3/2)L(s,\chi^{4n-3})
\end{equation}
around a contour enclosing these zeros.

We have, with the contour $C$ defined as a rectangle with corners
2, $2+iT$, $-1+iT$ and $-1$,
\begin{eqnarray}
&&\# \{ \gamma_n\leq T: L(1/2+i\gamma,\chi^{4n-3})=0\}\nonumber
\\&&\quad =
\frac{1}{2\pi}\Delta_C \arg (\xi(s,\chi^{4n-3}))\nonumber \\
&&\quad =\frac{1}{\pi} \Delta \arg\left(
\left(\frac{7}{2\pi}\right)^{s}
\Gamma(s+2n-3/2)\right)\big|_{s=1/2}^{s=1/2+iT} +O\left(\frac{\log 2n}{\log \log 3n}\right)\nonumber \\
&&\quad = \frac{T}{\pi} \log 2n  +O\left(\frac{\log 2n}{\log \log
3n}\right).
\end{eqnarray}
In the third line we consider just half the contour $C$ (from
$s=1/2$, through $s=2$ and $s=2+iT$, to $s=1/2+iT$) because the
functional equation gives
$\xi(\sigma+it)=\xi(1-\sigma-it)=\overline{\xi(1-\sigma+it)}$. The
change in argument of $L(s,\chi^{4n-3})$ is contained in the
error term and follows by standard methods on assuming the Riemann
Hypothesis for this $L$-function.  For the final line we use
Stirling's formula and so we see that the zeros, $\gamma_n$, of
$L(s,\chi^{4n-3})$ need to be scaled as
\begin{equation}
\widetilde{\gamma}_n=\frac{\log 2n}{\pi} \gamma_n
\end{equation}
in order to have approximate unit mean spacing.

We now define a new test function\footnote{For convenience we scale up by $\log N$ instead of the asymptotically equal $\log 2n$.}
\begin{equation}
\phi(x)=f\left(\frac{x\log N}{\pi}\right).
\end{equation}
We note that
\begin{eqnarray}
\hat{\phi}(x)&:=&\int_{-\infty}^{\infty} \phi(u)e(-ux)du \nonumber
\\&=& \int_{-\infty}^{\infty} f\left(\frac{u\log N}{\pi}\right)
e(-ux)
du\nonumber \\
&=& \frac{\pi}{\log N}\int_{-\infty}^{\infty} f(u)e\left(-\frac{u
\pi x}{\log N}\right) du \nonumber \\ &=&\hat{f}\left(\frac{\pi
x}{\log N}\right) \;\frac{\pi}{\log N}.
\end{eqnarray}

With the support of $\hat{f}$ restricted, ${\rm supp} \hat{f}
\subset [-\alpha,\alpha]$, $\alpha<1$,  we have a scaled explicit
formula:
\begin{eqnarray}\label{eq:efscaled}
&&\frac{1}{N}\sum_{n=1}^N\sum_{\gamma_n}f(\gamma_n \log N/\pi)
=\sum_{n=1}^N \left( -\frac{1}{2\pi N} \int_{-\infty}^{\infty}
\frac{{X}_{4n-3}'}{{X}_{4n-3}}(1/2+it) f\left(\frac{t\log N}{\pi}\right) dt
\right.\nonumber\\
&&\qquad\qquad \left.-\frac{1}{N\log N} \sum_{k=1}^{\infty}
\frac{\Lambda_{4n-3}(k)}{\sqrt{k}} \hat{f}\left(\frac{\log k} {2\log
N}\right)\right).
\end{eqnarray}

To evaluate the right side of this formula we start with
\begin{eqnarray}
\sum_{n=1}^N \frac{{X}_{4n-3}'}{{X}_{4n-3}}(1/2+it)&=&
\sum_{n=1}^N \Big(-\frac{\Gamma'}{\Gamma}(2n-1-it)-\frac{\Gamma'}{\Gamma}(2n-1+it)
-2\log \frac{7}{2\pi}\Big)\nonumber \\
&=&-2\sum_{n=1}^N \left( \log 2n +
O(1)\right)\nonumber \\
&=&-2N\log N +O(N)
\end{eqnarray}
for bounded $t$.
So we have
\begin{eqnarray}
\left(\frac{\log N}{\pi} +O(1)\right) \int_{-\infty}^{\infty}
f\left( \frac{t\log N}{\pi}\right) dt &=& \left(
1+O\left(\frac{1}{\log N}\right) \right) \int_{-\infty}^{\infty}
f(t)dt \nonumber \\
&=&\left( 1+O\left(\frac{1}{\log N}\right) \right)
\hat{f}(0)\nonumber \\
&\rightarrow& \hat{f}(0) \; {\rm as} \; N\rightarrow \infty.
\end{eqnarray}

Next we address the sum in (\ref{eq:efscaled}) containing
$\Lambda_{4n-3}(k)$.  With $\eta=(1+\sqrt{-7})/2$, we recall
\begin{equation}
L(s,\chi^{4n-3}) = \sum_{k=1}^{\infty}\frac{\chi^{(4n-3)}(k)}{k^{s+2n-3/2}} =
\frac{1}{2} \sum_{k=1}^{\infty} \sum_{a^2+ab+2b^2=k}
\epsilon_{a,b} \frac{(a+b\eta)^{4n-3}} {k^{s+2n-3/2}}.
\end{equation}
The Euler product is
\begin{equation}
L(s,\chi^{4n-3})= \prod_{p\neq 7} \left(
1-\frac{\chi^{(4n-3)}(p)}{p^{s+2n-3/2}}+\frac{p^{4n-3}}{p^{2s}}\right)^{-1}=\prod_{p\neq 7}
\left( 1-\frac{\alpha_n(p)}{p^s}\right)^{-1}
\left(1-\frac{\overline{\alpha_n(p)}}{p^s} \right)^{-1},
\end{equation}
where the $\alpha_p$ have the properties
\begin{eqnarray}
|\alpha_n(p)|&=&1;\nonumber \\
\frac{\chi^{(4n-3)}(p)}{p^{2n-3/2}}=\alpha_n(p)+\overline{\alpha_n(p)}&=&\frac{1}{2}\frac{\sum_{a^2+ab+2b^2=p}\epsilon
_{a,b} (a+b\eta)^{4n-3}} {p^{2n-3/2}};\nonumber \\
\frac{\chi^{(4n-3)}(p^2)}{p^{4n-3}}&=&\alpha_n(p)^2+1+\overline{\alpha_n(p)}^2.
\end{eqnarray}
Thus,
\begin{eqnarray}
\frac{L'}{L}(s,\chi^{4n-3})&=&\frac{d}{ds} \left(-\sum_{p\neq 7}
\log\left(1-\frac{\alpha_n(p)}{p^s}\right) +\log
\left(1-\frac{\overline{\alpha_n(p)}}{p^s}\right)\right) \nonumber\\
&=&\frac{d}{ds}\sum_{p\neq 7} \sum_{r=1}^\infty\left(
\frac{\alpha_n(p)^r}{rp^{rs}}
+\frac{\overline{\alpha_n(p)^r}}{rp^{rs}}\right)\nonumber \\
&=&-\sum_{p\neq 7} \log p \sum_{r=1}^\infty
\frac{(\alpha_n(p)^r+\overline{\alpha_n(p)}^r)}{p^{rs}},
\end{eqnarray}
and so by (\ref{eq:defLambda}) we have $\Lambda_{4n-3}(k)=0$ if $k$ is
not a prime power and
\begin{eqnarray}
\Lambda_{4n-3}(p^r) &= &\log p
\;(\alpha_n(p)^r+\overline{\alpha_n(p)}^r).
\end{eqnarray}

For all $k=p^r$, with $r\geq 3$, the sum over the coefficients
$\Lambda_{4n-3}(k)$ in (\ref{eq:efscaled}) is, with the sum over $p$
running over primes,
\begin{eqnarray}
&&\frac{1}{N\log N} \sum_{n=1}^N\sum_{p} \sum_{r\geq 3}\frac{\log
p(\alpha_n(p)^r+\overline{\alpha_n(p)}^r)}{p^{r/2}}
\hat{f}\left(\frac{r\log p} {2\log N}\right)\nonumber \\
&&\qquad \ll \frac{1}{\log N} \sum_{p} \frac{\log p}{p^{3/2}}
\nonumber \\
&&\qquad\ll\frac{1}{\log N}.
\end{eqnarray}
Here the second line follows because $\hat{f}$ and $\alpha_n(p)$
are bounded and once removed leave a geometric series in $r$ and
no dependence on $n$.

When $k=p$ is a prime, we have the sum
\begin{eqnarray}\label{eq:kprime}
&&\frac{-1}{2N\log N} \sum_{n=1}^N \sum_{p\leq N^{2\alpha}}
\hat{f}\left(\frac{\log p}{2 \log N}\right) \frac{\log
p}{\sqrt{p}} \sum_{a^2+ab+2b^2=p} \xi_{a,b} e(4n\theta_{a,b}),
\end{eqnarray}
where
\begin{equation}
\xi_{a,b}=\epsilon_{a,b} e(-3\theta_{a,b})
\end{equation}
and
\begin{equation}
e(\theta_{a,b})=\frac{a+b\eta}{\sqrt{a^2+ab+2b^2}}.
\end{equation}
In this case  $4\theta_{a,b}\notin \mathbb{Z}$ because
$e(\theta_{a,b})\neq 1,-1,i,-i$: if $b=0$ then $e(\theta_{a,b})
=\pm 1$ but then $a^2+0+0\neq p$ for any prime $p$ and if
$a+b/2=0$ then $e(\theta_{a,b})$ is purely imaginary, but
$a^2-2a^2+2(2a)^2=7a^2\neq p$ for any prime $p$. Consequently, as
in the proof of the first moment (see (\ref{eqn:sin}), we have
\begin{equation}\label{eq:p1}
\big| \sum_{n=1}^N e(4n\theta)\big| = \big|
\frac{e(4(N+1)\theta)-e(4\theta)} {e(4\theta)-1} \big| \leq
\frac{1}{|\sin 4\pi \theta|}.
\end{equation}
We then have the estimate
\begin{equation}\label{eq:p2}
\frac{1}{\sin 4\pi\theta_{a,b}} \ll \sqrt{a^2+ab+2b^2} \times \max
\left\{ \frac{1}{|a+b/2|},\frac{1}{|b|}\right\}.
\end{equation}
So, we approximate (\ref{eq:kprime}) with
\begin{eqnarray}
&&\frac{-1}{2N\log N}\sum_{a,b\leq N^{\alpha}}
\max\left\{\frac{1}{|a+b/2|},\frac{1}{|b|}\right\} \ll \frac{1}{N\log N} N^{\alpha} \log N.
\end{eqnarray}
In the first line above we have discarded the bounded quantities
$\hat{f}\big(\frac{\log p}{2\log N}\big)$ and $\xi_{a,b}$,
replaced the sum over $n$ of $e(4n\theta_{a,b})$ using
(\ref{eq:p1}) and (\ref{eq:p2}), and removed the requirement that
$a^2+2ab+2b^2$ be a prime; we allow it to be any integer but drop
the $\log p$ from the equation.  The final line comes from the sum
over $a$ and $b$.  The $1/|b|$ sum can be evaluated as
$\sum_{a=1}^{N^{\alpha}}1$ times
$\sum_{b=1}^{N^{\alpha}}\tfrac{1}{b}$, and the $1/|a+b/2|$ sum is
$\sum_{n=1}^{3N^{\alpha}} \frac{f(n)}{n/2}$, where $f(n)$, the
number of ways to obtain $a+b/2=n$ is certainly less than
$N^{\alpha}$, giving a result of $N^{\alpha}\log N$.

The final case is $k=p^2$ for all primes $p$.  Recall that
\begin{eqnarray}
\Lambda_{4n-3}(p^2)&=&\log p
\;(\alpha_n(p)^2+\overline{\alpha_n(p)}^2)\nonumber \\
&=& \log p \;\Big(\frac{\chi^{(4n-3)}(p^2)}{p^{4n-3}}-1\Big)
\end{eqnarray}
So, the relevant term from the explicit formula is
\begin{eqnarray}\label{eq:squares}
&&\frac{-1}{N\log N}\sum_{n=1}^N\sum_{p\leq N^\alpha}
\hat{f}\left(\frac{\log p}{\log N}\right) \frac{\log p}{p}\left(
\frac{1}{2} \sum_{a^2+ab+2b^2=p^2} \xi_{a,b} e(4n\theta_{a,b})
-1\right).
\end{eqnarray}
If we consider first just the term with the $-1$ in the final
bracket we get, using the Prime Number Theorem, and remembering
that $\hat{f}$ is even and has compact support in
$[-\alpha,\alpha]$,
\begin{eqnarray}
&& \frac{1}{\log N} \sum_{p\leq N^{\alpha}} \hat{f}\left(
\frac{\log p}{\log N}\right) \frac{\log p}{p} \sim \frac{1}{\log
N} \int_{1}^{N^{\alpha}} \hat{f}\left(\frac{\log u}{\log N}\right)
\frac{du}{u} \nonumber \\
&&\qquad\qquad =\int_{0}^{\alpha} \hat{f}(\beta) d
\beta=\frac{1}{2}  \int_{-1}^{1} \hat{f}(\beta)d\beta.
\end{eqnarray}
The remaining part of (\ref{eq:squares}) is
\begin{eqnarray}\label{eq:intornot}
&&\frac{-1}{2N\log N}\sum_{n=1}^N\sum_{p\leq N^\alpha}
\hat{f}\left(\frac{\log p}{\log N}\right) \frac{\log
p}{p}\sum_{a^2+ab+2b^2=p^2} \xi_{a,b} e(4n\theta_{a,b})
\\
&&=\frac{-1}{2N\log N}\sum_{p\leq N^\alpha}
\hat{f}\left(\frac{\log p}{\log N}\right) \frac{\log
p}{p}\sum_{a^2+ab+2b^2=p^2} \xi_{a,b}\left\{ \begin{array}{cc}
\frac{e(4(N+1)\theta_{a,b}) -e(4\theta_{a,b})}
{e(4\theta_{a,b})-1} & {\rm if\;\;} 4\theta_{a,b}\notin \mathbb{Z}\\
N & {\rm if\;\;} 4\theta_{a,b}\in
\mathbb{Z}.\end{array}\right.\nonumber
\end{eqnarray}
For $4\theta_{a,b}\notin \mathbb{Z}$, the contribution is
\begin{eqnarray}
&&\ll\frac{1}{N\log N} \sum_{a,b\leq N^{\alpha}} \frac{\max\left\{
\frac{1}{|a+b/2|},\frac{1}{|b|}\right\}}{\sqrt{a^2+ab+2b^2}}
\nonumber \\
&& \leq\frac{1}{N\log N} \sum_{a,b\leq N^{\alpha}} \frac{
\frac{1}{|a+b/2|}+\frac{1}{|b|}}{\sqrt{a^2+ab+2b^2}}
\nonumber \\
&& \ll\frac{1}{N\log N} \sum_{a,b\leq N^{\alpha}}
\frac{1}{|a+b/2|^2}+\frac{1}{|b|^2},
\end{eqnarray}
because $a^2+ab+2b^2=(a+b/2)^2+7b^2/4$.  Thus the
$4\theta_{a,b}\notin \mathbb{Z}$ contribution is
\begin{equation}
\ll \frac{1}{N\log N}{N^{\alpha}}.
\end{equation}

When $4\theta_{a,b}\in \mathbb{Z}$, then $b=0$, arguing as after
(\ref{eq:p2}).  The contribution to (\ref{eq:intornot}) is
\begin{eqnarray}\label{eq:intcase}
&& -\frac{1}{2\log N} \sum_{p\leq
N^{\alpha}}\hat{f}\left(\frac{\log p}{\log N}\right) \frac{\log
p}{p}\sum_{a=\pm p} \xi_{a,0}.
\end{eqnarray}
We see that
\begin{eqnarray}
\xi_{\pm p,0}&=&\epsilon_{\pm p,0}\; e(-3\theta_{\pm
p,0})\nonumber
\\
&=& \left(\frac{\pm p^3}{7}\right)\times \pm 1= \pm
\left(\frac{\pm p}{7}\right)=\left(\frac{p}{7}\right),
\end{eqnarray}
because  $\left(\frac{-1}{7}\right)=-1$.  Thus (\ref{eq:intcase})
becomes
\begin{eqnarray}
&& -\frac{1}{\log N}\sum_{p\leq N^{\alpha}}
\int_{-\infty}^{\infty}f(t) e\left(\frac{-t\log p}{\log N}\right)
dt \left(\frac{p}{7}\right) \frac{\log p}{p}\nonumber \\
&&\ll \frac{1}{\log N} \int_{-\infty}^{\infty} f(t) \sum_p\frac{
\left( \frac{p}{7}\right) \log p} {p^{1+\tfrac{2\pi i t}{\log N}}}
dt \nonumber \\
&&\ll\frac{1}{\log N} \int_{-\infty}^{\infty} f(t) \frac{L'}{L}
\left( 1+\frac{2\pi i t}{\log N}, \chi_{-7}\right) dt \ll
\frac{1}{\log N},
\end{eqnarray}
where the final approximation follows because $L'/L$ grows slowly
on the 1-line and $f(t)$ decays fast enough to keep the integral
bounded.

Thus for large $N$  (\ref{eq:efscaled}) becomes
\begin{eqnarray}
&&\frac{1}{N}\sum_{n=1}^N\sum_{\gamma_n}f(\gamma_n \log N/\pi)=
\hat{f}(0) +\frac{1}{2}  \int_{-1}^{1} \hat{f}(\beta)d\beta +o(1)
\nonumber \\
&&= \int_{-\infty}^{\infty} \hat{f}(\beta)\left(\delta(\beta)
+\frac{1}{2}  I_{[-1,1]}(\beta)\right) d\beta +o(1)\nonumber \\
&&=\int_{-\infty}^{\infty} f(y) \left( 1+\frac{\sin 2\pi y}{2\pi
y}\right)dy+o(1),
\end{eqnarray}
where $I_{[-1,1]}$ is the characteristic function on the interval
$[-1,1]$.  This is consistent with even orthogonal symmetry for
this family of $L$-functions: recall that for the group $SO(2N)$
the limiting form of the one-level density is
\begin{equation}
1+\frac{\sin 2\pi y}{2\pi y}
\end{equation}
and its Fourier transform is
\begin{equation}
\delta(u)+\frac{1}{2}I_{[-1,1]}(u).
\end{equation}

To derive the theorem we argue as in \cite{kn:ILS99}. Suppose that
our test function $f$ is non-negative, that $f(0)=1$ and that the
Fourier transform $\hat{f}$ is supported in $[-1,1]$. We write the
above equation as \begin{equation}
\frac{1}{N}\sum_{n=1}^N\sum_{\gamma_n}f(\gamma_n \log
N/\pi)=v+o(1).\end{equation} Let \begin{equation}p_m(N)=\frac 1 N
\#\{n\le N: \mbox{ the order of the zero of $L(s,\chi^{4n-3})$ at
$s=1/2$ is $m$}\}.\end{equation} Then
\begin{equation}\sum_{m=1}^\infty m p_m(N)\le
\frac{1}{N}\sum_{n=1}^N\sum_{\gamma_n}f(\gamma_n \log
N/\pi)=v+o(1).\end{equation} Clearly,
\begin{equation}\sum_{m=0}^\infty p_m(N)=1.\end{equation} Also, since all of the
$L(s,\chi^{4n-3})$ have even functional equations it follows that
$p_m(N)=0$ if $m$ is odd. Therefore,
\begin{eqnarray}
&&\sum_{m=1}^\infty m p_m(N) = 2p_2(N)+4p_4(N)+6p_6(N)+\dots
\nonumber \\ &&\qquad\ge 2p_2(N)+2p_4(N)+2p_6(N)+\dots = 2
-2p_0(N).
\end{eqnarray} Thus,
\begin{equation} \lim_{N\to \infty} p_0(N)\ge \frac{2-v}{2}.
\end{equation} With the
choice
\begin{equation}f(y)=\left(\frac {\sin \pi y}{\pi
y}\right)^2
\end{equation} we have
\begin{equation}\hat{f}(x)
=\max\{0,1-|x|\}
\end{equation} so that $v=3/2$. It follows that
$p_0(N)\ge 1/4-\epsilon.$ In words, at least one-fourth of the $L$
functions in our family do not vanish at their central points.

\section{One-level density from the ratios conjecture}

The previous section demonstrated a rigorous calculation of the
one-level density for the family of $L$-functions.  The
limitations of this technique is always the restricted support of
the Fourier transform of the test function $f(y)$.  In the
following we remove the restriction on the support, but the result
is conditional on a ratios conjecture of the type introduced by
Conrey, Farmer and Zirnbauer \cite{kn:cfz2}.  We are interested in
a conjecture for the quantity

\begin{eqnarray}
&&\sum_{n=1}^N
\frac{L(1/2+\alpha,\chi^{4n-3})}{L(1/2+\gamma,\chi^{4n-3})}.
\end{eqnarray}

We use the notation:
\begin{eqnarray}
L(s,\chi^{4n-3})&=& \prod_{p\neq 7} \left( 1-
\frac{\alpha_n(p)}{p^s}\right)^{-1} \left(
1-\frac{\overline{\alpha_n}(p)} {p^s}\right)^{-1} =
\sum_{m=1}^\infty \frac{a_n(m)}{m^s} \\
\frac{1}{L(s,\chi^{4n-3})}&=&\prod_{p\neq 7} \left( 1-
\frac{\alpha_n(p)}{p^s}\right)\left(
1-\frac{\overline{\alpha_n}(p)} {p^s}\right) = \sum_{m=1}^\infty
\frac{\mu_n(m)}{m^s},
\end{eqnarray}
so
\begin{equation}
\mu_n(p^\ell)=\left\{ \begin{array}{cc} 1&{\rm if\;\;}
\ell=0\\-a_n(p) & {\rm if
\;\;} \ell=1\\ 1 & {\rm if \;\;} \ell=2\\
0&{\rm if\;\;} \ell\geq 3\end{array}\right. .
\end{equation}
Note that $a_n(m)=\chi^{(4n-3)}(m)/m^{2n-3/2}$ in our previous notation.
We are interested in averages of the form
\begin{eqnarray}
\delta_\mu(p^m,p^\ell)&:=& \langle a_n(p^m)\mu_n(p^\ell)\rangle
=\lim_{N\to \infty} \frac 1 N \sum_{n=1}^N a_n(p^m)\mu_n(p^\ell)
\end{eqnarray}
It can be shown that
\begin{eqnarray}
\delta_\mu(p^m,p^\ell) &=&\left\{ \begin{array}{cl} 0 & {\rm
if\;\;}m {\rm \;\;odd,\;\;} \ell=0,2\\1&{\rm if\;\;} m {\rm
\;\;even,\;\;} \ell=0,2, \;p\equiv 1,2,4\mod
7\\
(-1)^{m/2} &{\rm if\;\;} m {\rm \;\;even,\;\;} \ell=0,2,\; p\equiv
3,5,6\mod 7\\
-2 & {\rm if\;\;} m {\rm
\;\;odd,\;\;} \ell=1,\; p\equiv 1,2,4\mod 7\\
0& {\rm if\;\;} m {\rm\;\;even,\;\;} \ell=1, \;p\equiv 1,2,4\mod
7\\0& {\rm if\;\;} \ell=1, \;p\equiv 3,5,6\mod 7\\ 0& {\rm if\;\;}
\ell\geq 3\\0& {\rm if\;\;} p=7\end{array}\right. .
\end{eqnarray}
 Following the ratios conjecture recipe, we take the contribution
to the $L$-function in the numerator from the first term in the
approximate functional equation:
\begin{equation}
L(s,\chi^{4n-3})=\sum_{m<x}\frac{a_n(m)}{m^s}+X_{4n-3}(s)\sum_{m<y}
\frac{a_n(m)}{m^{1-s}}.
\end{equation}
Thus we are interested in the sum
\begin{eqnarray}\label{eq:firstterma}
&& \sum_{n=1}^N \left[ \sum_m
\frac{a_n(m)}{m^{1/2+\alpha}}\;\sum_{\ell} \frac{\mu_n(\ell)}
{\ell^{1/2+\gamma}}\right]\nonumber \\
&& = \sum_{m,\ell}
\frac{1}{m^{1/2+\alpha}}\frac{1}{\ell^{1/2+\gamma}} \left(
\sum_{n=1}^N a_n(m)\mu_n(\ell)\right)\nonumber \\
&&=\prod_p \sum_{m,\ell}
\frac{\delta_\mu(p^m,p^\ell)}{p^{(1/2+\alpha)m}p^{(1/2+\gamma)\ell}}
\nonumber \\
&&=\prod_{p\equiv 1,2,4\mod 7} \left(
-\frac{2}{p^{1+\alpha+\gamma}} \left( \sum_{m=0}^\infty
\frac{1}{p^{(1+2\alpha)m}}\right) + \left(
1+\frac{1}{p^{1+2\gamma}}\right) \sum_{m=0}^\infty
\frac{1}{p^{(1+2\alpha)m}}\right)\nonumber \\
&&\qquad \times \prod_{p\equiv 3,5,6\mod 7} \left( \left(
1+\frac{1}{p^{1+2\gamma}}\right)\left( \sum_{m=0}^\infty
\frac{(-1)^m}{p^{(1+2\alpha)m}}\right)\right).
\end{eqnarray}
Now we want to extract from (\ref{eq:firstterma}) all the factors
that don't converge like $\tfrac{1}{p^2}$.  Thus, the limit as $N\to \infty$ of the expression in (\ref{eq:firstterma}) divided by $N$ equals
\begin{equation}
\frac{\zeta(1+2\gamma)\;L(1+2\alpha,\chi_{-7})}
{\zeta(1+\alpha+\gamma)\;L(1+\alpha+\gamma,\chi_{-7})} \times
A(\alpha,\gamma),
\end{equation}
with the convergent Euler product $A(\alpha,\gamma)$ given by
\begin{eqnarray}\label{eq:arithA}
&&A(\alpha,\gamma)= \prod_p
\frac{\left(1-\frac{1}{p^{1+2\gamma}}\right)}{
\left(1-\frac{1}{p^{1+\alpha+\gamma}}\right)} \nonumber \\
&&\qquad\times \prod_{p\equiv 1,2,4\mod 7} \frac{\left(
1-\frac{1}{p^{1+2\alpha}}\right)}
{\left(1-\frac{1}{p^{1+\alpha+\gamma}}\right)} \left(
-\frac{2}{p^{1+\alpha+\gamma}} \left( \sum_{m=0}^\infty
\frac{1}{p^{(1+2\alpha)m}}\right) + \left(
1+\frac{1}{p^{1+2\gamma}}\right) \sum_{m=0}^\infty
\frac{1}{p^{(1+2\alpha)m}}\right) \nonumber \\
&&\qquad \times \prod_{p\equiv 3,5,6\mod 7} \frac{\left(
1+\frac{1}{p^{1+2\alpha}}\right)}
{\left(1+\frac{1}{p^{1+\alpha+\gamma}}\right)}\left( \left(
1+\frac{1}{p^{1+2\gamma}}\right)\left( \sum_{m=0}^\infty
\frac{(-1)^m}{p^{(1+2\alpha)m}}\right)\right).
\end{eqnarray}

Using the second sum in the approximate functional equation merely
exchanges $-\alpha$ for $\alpha$ and introduces a factor
$X_{4n-3}(1/2+\alpha)$.  Thus we arrive at the conjecture, making the
usual assumptions about the error term and applying the usual
restrictions on the parameters $\alpha$ and $\gamma$,

\begin{conjecture}\label{conj:oneratio}
Let $-1/4<\Re \alpha<1/4$, $1/\log N\ll \Re \gamma<1/4$ and $\Im
\alpha,\Im \gamma\ll_\epsilon N^{1-\epsilon}$, then following the
Ratios Conjecture recipe \cite{kn:cfz2} we have
\begin{eqnarray}
&&\sum_{n=1}^N
\frac{L(1/2+\alpha,\chi^{4n-3})}{L(1/2+\gamma,\chi^{4n-3})} =
\sum_{n=1}^N \left(
\frac{\zeta(1+2\gamma)\;L(1+2\alpha,\chi_{-7})}
{\zeta(1+\alpha+\gamma)\;L(1+\alpha+\gamma,\chi_{-7})} \times
A(\alpha,\gamma)\right.\nonumber
\\&&\qquad\qquad\left.+\left(\frac{7}{2\pi}\right)^{-2\alpha}
\frac{\Gamma(2n-1-\alpha)}{\Gamma(2n-1+\alpha)}
\frac{\zeta(1+2\gamma)\;L(1-2\alpha,\chi_{-7})}
{\zeta(1-\alpha+\gamma)\;L(1-\alpha+\gamma,\chi_{-7})} \times
A(-\alpha,\gamma)\right)\nonumber \\
&&\qquad\qquad +O(N^{1/2+\epsilon}).
\end{eqnarray}
with $A(\alpha,\gamma)$ defined at (\ref{eq:arithA}).
\end{conjecture}

We note that in arriving at this conjecture we have used the
ratios conjecture method exactly as was done for ratios of
quadratic twists of the $L$-function associated with Ramanujan's
tau-function, $L_\Delta(s,\chi_d)$, Conjecture 2.9 of
\cite{kn:consna06}. That conjecture and Conjecture
\ref{conj:oneratio} have exactly the same structure and in a
straightforward analogy with the lines leading up to Theorem 2.10
in \cite{kn:consna06} and Theorem 3.2 in that paper we see without
need to include further workings, that the form of the one-level
density that follows from the ratios conjecture
\ref{conj:oneratio} has the form
\begin{theorem}
Assuming Conjecture \ref{conj:oneratio} and assuming for
simplicity that $f(z)$ is even, holomorphic in the strip $|\Im
z|<2$, real on the real line and that $f(x)\ll 1/(1+x^2)$ as
$x\rightarrow \infty$,
\begin{eqnarray}
\frac{1}{N} \sum_{n=1}^N \sum_{\gamma_n} f(\gamma_n) &=&
\frac{1}{2\pi N} \int_{-\infty}^{\infty}f(t) \sum_{n=1}^N \left(
2\log \frac{7}{2\pi} +\frac{\Gamma'}{\Gamma}
(2n-1+it)+\frac{\Gamma'}{\Gamma}(2n-1-it)\right.\nonumber \\
&&+2\left( -\frac{\zeta'(1+2it)}{\zeta(1+2it)} +
\frac{L'(1+2it,\chi_{-7})}{L(1+2it,\chi_{-7})} + A'(it, it)\right.\nonumber\\
&&\left.\left.-\left(\frac{7}{2\pi}\right)^{-2it}
\frac{\Gamma(2n-1-it)}{\Gamma(2n-1+it)} \frac{\zeta(1+2it)
L(1-2it,\chi_{-7})} {L(1,\chi_{-7})} A(-it,it)\right)\right)dt\nonumber \\
&&+O(N^{-1/2+\epsilon}),
\end{eqnarray}
where $A'(r,r)$ is defined as
\begin{equation}
\frac{d}{d\alpha}A(\alpha,\gamma)\big|_{\alpha=\gamma=r}.
\end{equation}
\end{theorem}

If the zeros are scaled by $(1/\pi)\log N$ so as to have,
asymptotically, unit density, then we see as in \cite{kn:consna06}
that
\begin{equation}
\lim_{N\rightarrow \infty} \frac{1}{N} \sum_{n=1}^N
\sum_{\gamma_n} f(\gamma_n \log N/\pi)=\int_{-\infty}^\infty
f(y)\left(1+\frac{\sin 2\pi y}{2\pi y} \right) dy.
\end{equation}
Here there is no restriction on the support of the Fourier
transform of the test function, but the result relies on the
correct form of the main terms in Conjecture \ref{conj:oneratio}.

\section{Discretization}
Our family of $L$-functions seems to have orthogonal symmetry
type, which means that it is potentially possible to have lots of
small values or zeros of these $L$-functions at the critical
point. For modeling this family with Random Matrix Theory we need
to know how the family is ``discretized.'' In practice, the
central values of the $L$-functions of an orthogonal family can be
given by a nice conjectural formula which expresses the value in
terms of the arithmetic, or geometry of the family. For example,
if $d$ is a fundamental discriminant, and $\chi_d(n)$ the
associated quadratic Dirichlet character, then the  central values
of the $d$th quadratic twist of an elliptic curve $L$-function is
\begin{equation}L_E(1/2,\chi_d)=\kappa \frac{(c(|d|))^2}{\sqrt{|d|}}\end{equation} where
$c(|d|)$ is an integer. Thus, we know that if the $L$-function has
a value smaller than $\kappa/\sqrt{|d|}$ then it must be 0. This
is what we mean by the discretization. In this situation, it
appears that the $c(|d|)$ may take any integral values subject to
some mild conditions arising from consideration of Tamagawa
factors that lead to some powers of 2 that must divide $c(|d|)$,
depending on the primes that divide $d$.

Our situation here is different. If we try to discretize in a
similar way, we need a formula for the central values and we need
to know just  how small our values $L(1/2,\chi^{4n-3})$ can be
without being 0. In this case, Rodriguez-Villegas and Zagier \cite{kn:vilzag93} have proven a formula, conjectured by Gross and Zagier \cite{kn:grozag80},
for the central value of the $L(s,
\chi^{2n-1})$, namely
\begin{equation} L(1/2, \chi^{2n-1})=
2\frac{(2\pi/\sqrt{7})^n\Omega^{2n-1}A(n)}{(n-1)!}
\end{equation}
 where
\begin{equation}\Omega= \frac{\Gamma(1/7)\Gamma(2/7)\Gamma(4/7)}{4\pi^2}=0.81408739831\dots . \end{equation}
By the functional equation $A(n)=0$ whenever $n$ is even. For odd $n$ Gross and Zagier \cite{kn:grozag80} conjectured that $A(n)$ is a square
 and gave the following table (in the notation of Rodriguez-Villegas and Zagier):
\begin{center}
\begin{tabular}{|l|l|l|}
\hline
$n$& $A(n)$&$L(1/2,\chi^{2n-1})$\\
\hline
1& 1/4&0.9666\\
3&1&4.7890\\
5&1&0.9885\\
7&$3^2$ & 0.7346\\
9&$7^2$ & 0.1769\\
11& $(3^2\cdot 5\cdot 7)^2$& 9.8609\\
13&$(3\cdot 7\cdot 29)^2$& 0.6916\\
15 & $(3\cdot 7 \cdot 103)^2$& 0.1187\\
17 & $(3\cdot 5\cdot 7\cdot 607)^2$& 1.0642\\
19 & $(3^3\cdot7\cdot 4793)^2$ &1.7403\\
21 & $(3^2\cdot5\cdot7\cdot29\cdot2399)^2$& 6.6396\\
23& $(3^3\cdot 5\cdot 7^2\cdot 10091)^2$ & 0.3302\\
25 & $(3^2\cdot 7^2\cdot 29\cdot 61717)^2$ & 0.2072\\
27&$(3^2\cdot5^2\cdot7^2\cdot13\cdot 53^2 \cdot 79 )^2$ & 1.2823\\
29 & $(3^4\cdot5^2\cdot7^2\cdot 113\cdot 127033)^2$ & 8.4268\\
31 & $(3^5\cdot 5\cdot 7^2 \cdot 71\cdot 1690651)^2$ & 0.6039\\
33 & $(3^4\cdot 5\cdot 7^2\cdot 1291\cdot 1747169)^2$ &0.0591\\
\hline
\end{tabular}
 \end{center}
Rodriguez-Villegas and Zagier \cite{kn:vilzag93} proved that $A(n)=B(n)^2$ where $B(1)=1/2$ and $B(n)$ is an integer for $n>1$.
In fact they prove a remarkable recursion formula:

Define sequences of polynomials $a_{k}(x)$, $b_k(x)$ by the recursions
\begin{eqnarray*}
a_{k+1}(x) =\sqrt{(1+x)(1-27 x)}
\left( x\frac{d}{dx} -\frac{2k+1}{3}\right) a_k(x) -\frac{k^2}{9}
(1-5x) a_{k-1}(x)
\end{eqnarray*}
and
\begin{eqnarray*}
21 b_{k+1}(x) =\left( (32 k x -56 k+42)-(x-7) (64x-7) \frac{d}{dx}\right)
b_k(x)-2k(2k-1)(11x+7)b_{k-1}(x)
\end{eqnarray*}
with initial conditions
$a_0(x)=1$, $a_1(x)=-\frac13 \sqrt{(1-x)(1+27x)}$, $b_0(x)=1/2$,
and $b_1(x)=1$. Then, with $A$ and $B$ defind as above 
\begin{eqnarray}
A(2n+1) =\frac{a_{2n}(-1)}{4} \qquad \mbox{and} \qquad 
B(2n+1) = b_n(0).
\end{eqnarray}
Equation (6) of \cite{kn:vilzag93} states that for odd $n$ 
\begin{eqnarray}
B(n)\equiv -n \bmod 4,
\end{eqnarray}
a result that in one fell swoop proves the non-vanishing of $L(1/2,\chi^{2n-1})$
for all odd $n$.

 It would be interesting to use these recursion 
formulae to try to understand a discretization of the values of this family of L-functions, 
from which one might profitably apply a random matrix model to infer more detailed statistical behavior
of these values. The integers $B(n)$ that appear in the formula of Villegas-Zagier
are growing quickly, presumably to counteract, by virtue of the expected Lindel\"{o}f Hypothesis, the $C^n(n-1)!$ growth in the denominator.
The question of how just how small these $L$-values can be   is an interesting one.

\section{Conclusion}
With a mind toward modeling the symmetric powers of the L-function of a general elliptic curve, and to 
consider whether their central values vanish, we have 
taken a few steps toward the much simpler problem of modeling the   family of 
L-functions associated with the symmetric powers of the L-function of an elliptic curve with complex multiplication.
We have used generally applicable methods from analytic number theory
even though for our particular family there are powerful algebraic 
methods available. 
We have given some evidence that the family has orthogonal symmetry type. 
Some unresolved questions are to  (1) Try to determine if the symmetric power
$L$-functions in the non-CM case form a family. This seemed very
nebulous to us, but see the letter \cite{kn:sarnak} of Sarnak to
Mazur for some interesting calculations relevant to this issue.
(2) Try to determine a discretization for the central values of an
$L$-function associated with a weight $k$ newform. This seems to
be a whole new direction that has not been considered. In
particular, for the Gross-Zagier family we consider, explaining
the genesis of the large values of the integers $A(n)$ that appear
as factors in the central values is a challenge. (3) Try to obtain
an asymptotic formula for the second moment of our family. Our
upper bound fell just short of achieving an asymptotic formula.
(4) Prove a bound of the form $L(1/2,\chi^{4n-3})\ll n^{1/2-\lambda}$ 
for some $\lambda>0$ (i.e. a subconvexity bound) for this family.

\newcommand{\etalchar}[1]{$^{#1}$}


\end{document}